\documentclass[a4paper]{amsart}

\usepackage[T1]{fontenc}
\usepackage{hyperref}
	\hypersetup{colorlinks=true,linkcolor=blue,citecolor=blue,filecolor=magenta,urlcolor=blue}      
\usepackage{geometry}       
\usepackage[english]{babel}  
\usepackage{amsmath,amsfonts,amssymb,amsthm,mathrsfs}
\usepackage{dsfont}
\usepackage{tikz,tkz-graph,tikz-cd}
\usepackage{enumerate}
\usepackage{enumitem}
\usepackage{setspace}

\newtheorem{theorem}{Theorem}[section]
\newtheorem{proposition}[theorem]{Proposition}
\newtheorem{lemma}[theorem]{Lemma}
\newtheorem{corollary}[theorem]{Corollary}

\theoremstyle{definition}
\newtheorem{definition}[theorem]{Definition}

\theoremstyle{remark}
\newtheorem{rmk}[theorem]{Remark}

\newcommand{\CC}{\mathds{C}}
\newcommand{\PP}{\mathds{P}}
\newcommand{\RR}{\mathds{R}}
\newcommand{\NN}{\mathds{N}}
\newcommand{\ZZ}{\mathds{Z}}
\newcommand{\QQ}{\mathds{Q}}
\newcommand{\BB}{\mathds{B}}

\newcommand{\A}{\mathcal{A}}
\newcommand{\B}{\mathcal{B}}

\newcommand{\E}{\mathcal{E}}
\newcommand{\I}{\mathcal{I}}
\newcommand{\M}{\mathcal{M}}
\newcommand{\ML}{\mathcal{ML}}
\newcommand{\N}{\mathcal{N}}
\newcommand{\R}{\mathcal{R}}
\newcommand{\W}{\mathcal{W}}
\newcommand{\Z}{\mathcal{Z}}

\renewcommand{\L}{\mathscr{L}}
\newcommand{\Lb}{\overline{\mathscr{L}}}

\renewcommand{\epsilon}{\varepsilon}
\newcommand{\PtoL}{ {(P\rightarrow L)} }
\newcommand{\inv}{^{-1}}
\newcommand{\ssnum}{ {\arabic{subsection}} }
\renewcommand{\hat}{\widehat}

\DeclareMathOperator{\HH}{H}
\DeclareMathOperator{\Sing}{Sing}
\DeclareMathOperator{\TLG}{TLG}
\DeclareMathOperator{\Id}{Id}
\DeclareMathOperator{\ulk}{ulk}

\DeclareMathOperator{\Aut}{Aut}
\DeclareMathOperator{\Gal}{Gal}
\DeclareMathOperator{\PGL}{PGL}
\DeclareMathOperator{\dist}{dist}
\DeclareMathOperator{\tub}{Tub}

\newcommand{\ocup}{\stackrel{\circ}{\cup}}
\newcommand{\Warning}{
\noindent
\tikz[
		 scale=1.7,
     line cap=but,
     line join=round,
     x=.5em,
		 line width=1.7pt,
     y=1*(height("Z")-\pgflinewidth)*(1-sin(10)),
     rotate=-15,
     rounded corners=1.5pt,
]\draw (1, 0) -- (0, 0) -- (1, 1) -- (0, 1);
\ 
}

\newcommand{\nocrossing}[2]{
\foreach \k in {1,...,#1}
	{\draw (#2,\k) -- (1+#2,\k);}
}

\newcommand{\undercrossing}[3]{
\foreach \k in {1,...,#1}
	{\draw (#2,\k) -- (1+#2,\k);}
\draw[color=white, very thick] (#2,#3) -- (1+#2,#3);
\draw[color=white, very thick] (#2,1+#3) -- (1+#2,1+#3);
\draw (#2,#3) to[out=0,in=180] (1+#2,1+#3);
\draw[color=white,line width=3.5pt] (#2,1+#3) to[out=0,in=180] (1+#2,#3);
\draw (#2,1+#3) to[out=0,in=180] (1+#2,#3);
}

\newcommand{\overcrossing}[3]{
\foreach \k in {1,...,#1}
	{\draw (#2,\k) -- (1+#2,\k);}
\draw[color=white, very thick] (#2,#3) -- (1+#2,#3);
\draw[color=white, very thick] (#2,1+#3) -- (1+#2,1+#3);
\draw (#2,1+#3) to[out=0,in=180] (1+#2,#3);
\draw[color=white, line width=3.5pt]  (#2,#3) to[out=0,in=180] (1+#2,1+#3);
\draw (#2,#3) to[out=0,in=180] (1+#2,1+#3);
}

\newcommand{\actualcrossing}[4]{
\foreach \k in {1,...,#1}
	{\draw (#2,\k) -- (-1+#4+#2,\k);}
\foreach \k in {1,...,#4}
	{\draw[color=white, very thick] (#2,#3+\k-1) -- (-1+#4+#2,#3+\k-1);}
\foreach \k in {1,...,#4}
	{\draw (#2,\k+#3-1) to[out=0,in=180] (-1+#4+#2,-\k+#4+#3);}
	
}

\begin{document}

\title{The loop linking number of line arrangements}

\author[B. Guerville-Ball\'e]{Beno\^it Guerville-Ball\'e}
\address{
Institute of Mathematics, Polish Academy of Sciences,
ul. Sniadeckich 8, 00-656 Warsaw, Poland
}
\email{benoit.guerville-balle@math.cnrs.fr }



\begin{abstract}
		In his Ph.D. thesis, Cadegan-Schlieper constructs an invariant of the embedded topology of a line arrangement which generalizes the $\I$-invariant introduced by Artal, Florens and the author. This new invariant is called the \emph{loop linking number} in the present paper. We refine the result of Cadegan-Schlieper by proving that the loop linking number is an invariant of the homeomorphism type of the arrangement complement. 
		
		We give two effective methods to compute this invariant, both are based on the braid monodromy. As an application, we detect an arithmetic Zariski pair of arrangements with 11 lines whose coefficients are in the $5$th cyclotomic field. Furthermore, we also prove that the fundamental groups of their complements are not isomorphic; it is the Zariski pair with the fewest number of lines which have this property. We also detect an arithmetic Zariski triple with 12 lines whose complements have non-isomorphic fundamental groups. In the appendix, we give 28 similar arithmetic Zariski pairs detected using the loop linking number.
	
		To conclude this paper, we give a multiplicativity theorem for the union of arrangements. This first allows us to prove that the complements of Rybnikov's arrangements are not homeomorphic, and then leads us to a generalization of Rybnikov's result. Lastly, we use it to prove the existence of homotopy-equivalent lattice-isomorphic arrangements which have non-homeomorphic complements.
\end{abstract}

\maketitle



\section*{Introduction}
	
A \emph{line arrangement} $\A$ is a finite set of lines $L_1,\dots,L_n$ in the complex projective plane $\CC\PP^2$. Such objects are at the intersection between \emph{hyperplane arrangements} and \emph{algebraic plane curves} (we refer to~\cite{OrlTer} and~\cite{BriKno} for an overview of these respective topics). This particular situation provides a great interest to the study of line arrangements. 

The first datum associated to a line arrangement $\A$ is its \emph{intersection lattice}, also called the \emph{combinatorics} of $\A$. Roughly speaking, this lattice describes how the lines of $\A$ intersect each other (for details see~\cite{OrlTer}). The second datum related to an arrangement $\A$ is the homeomorphism type of the pair $(\CC\PP^2,\bigcup_{L\in\A}L)$. This type is called the \emph{embedded topology} of $\A$, or simply its \emph{topology}.  It is accompanied by various questions about the homology or the homotopy of the complement $M(\A)=\CC\PP^2\setminus\bigcup_{L\in\A}L$ of~$\A$; the central question in the study of the embedded topology of an arrangement concerns its relationship with the intersection lattice. It is straighforward that two topologically equivalent arrangements have isomorphic intersection lattices\footnote{We should also mention here the refinement obtained by Jiang and Yau~\cite{JiaYau} who proved that the homeomorphism type of the complement determines the intersection lattice.}. It is therefore natural to investigate the reciprocal: is the embedded topology of an arrangement determined by the intersection lattice?

The earliest major result in this investigation is due to Orlik and Solomon in the 80's~\cite{OS}. They proved that the cohomology ring of the complement of an hyperplane arrangement is fully determined by its intersection lattice. In the opposite direction, Rybnikov provided~\cite{Ryb}, at the end of the 90's, an explicit example of a pair of arrangements which have isomorphic intersection lattices and whose complements have non-isomorphic fundamental groups (and have therefore different topologies). Such pairs (i.e. with isomorphic combinatorics and different topologies) are called \emph{Zariski pairs} in the literature\footnote{This name arises from the work of Zariski who pioneered the study of the topology of equisingular curves~\cite{Zar,Zar1}. In particular, he constructed the first exemple of two irreducible curves of degree 6 which have the same local singularities and non-isomorphic fundamental groups.}. This result leads to the necessity of understanding this gap between the combinatorics and the topology of line arrangements. It also opens the general question: "\emph{If $\A$ is an arrangement, which invariant of its topology is determined by its combinatorics?}" For example, the question is still open for the Betti numbers or the monodromy of the Milnor fiber of $\A$ (see~\cite[Problem~4.5]{FalkRan} and~\cite[Problem~5.12]{Suc:Milnor} respectively), or for the characteristic varieties of $\A$~\cite[Problem~3.15]{Suc:Milnor}.

In the decade that followed Rybnikov's result, only one other Zariski pair has been found~\cite{ACCM} (as for the Rybnikov example, a computer assitance was required). In parallel, the same team composed of Artal, Carmona, Cogolludo and Marco made significant progress in understanding Rybnikov's example~\cite{ACCM:Rybnikov}; and the combined work of Nazir and Yoshinaga~\cite{NazYos}, and of Fei~\cite{Fei} proved that there is no Zariski pair up to nine lines (leaving open the question of the existence of Zariski pairs with 10 lines).

During the last decade, questions like the combinatoriality of the fundamental group of real complexified arrangements~\cite[Problem~1.3]{FalkRan} or the existence of arithmetic Zariski pairs\footnote{In the general case of algebraic plane curves, a Zariski pair is said to be \emph{arithmetic} if the equations of their arrangements are conjugated in a Galois field. Such pairs have been extensively studied by authors like A. Degtyarev, M. Oka, I. Shimada or H. Tokunaga. Note that the Zariski pairs of arrangements given in~\cite{ACCM,Gue:ZP} are arithmetic.} with non-isomorphic fundamental groups of the complements (discussed in~\cite{ArtCogTok}) have been solved in~\cite{ABGBVS} for the former and in~\cite{ACCM} for the later. Another noticeable step is the first proof, produced without computer assistance, of the existence of a Zariski pair~\cite{GBVS}. We should also mention the discovery of 11 new Zariski pairs~\cite{Gue:ZP,GBVS} or the first family of arithmetic Zariski tuple~\cite{Cad}. All these recent results arise from two papers: the construction of a \emph{linking invariant} for line arrangements~\cite{AFG} (called the $\I$-invariant), and the description of the inclusion of the boundary manifold\footnote{The boundary manifold is the boundary of a regular tubular neighborhood of $\bigcup_{L\in\A}L$. It is a graph $3$-manifold determined by the combinatorics of $\A$. This manifold has been studied in~\cite{CohSuc:Boundary_manifold}.} in the complement of $\A$~\cite{FGM} (which allows to compute the $\I$-invariant). Therefore, the study of the linking properties of line arrangements and the resulting Zariski pairs could help solve open problems like the combinatoriality of the characteristic varieties or the questions related to the Milnor fiber.\\

The $\I$-invariant was thought as an adaptation of linking numbers in Knot Theory, and was therefore constructed from a geometric-topological point of view. In his Ph.D. thesis~\cite{Cad}, Cadegan-Schlieper gives a more "algebraic-topological" flavor to the definition of the $\I$-invariant, and uses this new approach to generalize it. This new linking invariant is called, in this paper, the \emph{loop linking number}. He uses this new invariant to prove that, for a fixed $N$, the arrangements associated to the complex reflection group\footnote{See~\cite{SheTodd} for more details about complex reflection groups.} $G(N,N,3)$ form an ordered and oriented Zariski tuple. Unfortunately, Cadegan-Schlieper quitted academia at the end of his Ph.D. and stopped his research on this topic. The purpose of this paper is therefore to investigate more deeply this new linking invariant, to provide a more geometric/topological point of view and to obtain new results and applications. \\

The loop linking number is constructed as follows. Let $G$ be an Abelian group and consider a tensor element $\Lambda$ in $\HH^1(M(\A) ; G)\otimes_\ZZ \HH_1(\Gamma(\A),\ZZ)$, where $\Gamma(\A)$ is the incidence graph of $\A$. Then let $i\circ j$ be the composition of an embedding $j$ of $\Gamma(\A)$ in the boundary manifold $B(\A)$, which respects the graph structure of $B(\A)$, with the embedding $i:B(\A)\hookrightarrow M(\A)$. We denote by $\Psi$ the map $\Id\otimes (i\circ j)_*$ which sends the previous tensor space in $\HH^1(M(\A) ; G) \otimes_\ZZ \HH_1(M(\A) , \ZZ)$. Let $\pi$ be the natural pairing of the latter tensor space; if $\Lambda$ verifies some combinatorial conditions, then the \emph{loop linking number of $\A$ associated to the tensor $\Lambda$}, defined as $\L(\A,\Lambda)=\pi\circ\Psi(\Lambda)$, is an invariant of the ordered and oriented topology\footnote{The ordered (resp. oriented) topological type of $\A$ is the class of homeomorphism of $\CC\PP^2$ which fixes $\bigcup_{L\in\A}L$ and respects a fixed complete order on $\A$ (resp. the global orientation of $\CC\PP^2$ and the local orientation of the meridians).} of $\A$, see~\cite[Proposition 21]{Cad}. It generalizes the $\I$-invariant, in the sense that this former linking invariant is given by $\I(\A,\xi,\gamma)=\L(\A,\xi\otimes\gamma)$, when $(\A,\xi,\gamma)$ is an inner-cyclic triple (see~\cite{AFG} for the definition).

The first result of this paper (see Theorem~\ref{thm:invariance_complement}) proves that the loop linking number is not only an invariant of the ordered and oriented topological type of $\A$, but also an invariant of the ordered and oriented homeomorphism type of its complement $M(\A)$. In order to remove the ordered and oriented conditions in the previous theorem, we introduce the \emph{full loop linking number} and prove, in Theorem~\ref{thm:full_loop_linking}, that it is an invariant of the homeomorphism type of the complement of~$\A$. Note that at the end of the paper, in Corollary~\ref{cor:not_homotopy_determined}, we prove that the loop linking number and the full loop linking number are not determined by the homotopy type of the complement.

Then, we give a geometrical interpretation of the loop linking number which leads to a formula for $\L(\A,\Lambda)$ based on the braid monodromy (see Remark~\ref{rmk:braid_monodromy} and Theorem~\ref{thm:computation}). From this description, two effective computational methods can be derived: the first is based on the braided wiring diagram (introduced by Arvola~\cite{Arvola}, see also~\cite{Good,CohSuc:braid_monodromy}), and the second uses change of variables to simplify the computation of the braid monodromy.\\
 
As an application of the previously mentioned methods, we compute the loop linking number of $4$ arrangements of 10 lines conjugated in the $5$th cyclotomic field (see Section~\ref{sec:oZP}), and we deduce in Corollary~\ref{cor:oZP}, that some of the pairs formed by these arrangements are ordered and oriented Zariski pairs. To remove the "ordered and oriented" condition, we follow the strategy of~\cite{ACCM} and~\cite{Gue:ZP}, and we add an 11th line to these arrangements which trivializes their combinatorics automorphism groups. This leads to the second Zariski pair of 11 lines (the former being the one of~\cite{ACCM}), which is also an arithmetic pair. In opposition with the former example, we know that the complements of these arrangements are not homeomorphic (see Theorem~\ref{thm:ZP}). Furthermore, applying the arguments of~\cite{ACGM}, we also prove in Theorem~\ref{thm:grp_fonda} that the fundamental groups of the complements of these arrangements are not isomorphic. Thus we provide the Zariski pair with the fewest number of lines which have non-isomorphic fundamental groups. As a second application and following the same construction as for the previous Zariski pair, we detect an arithmetic Zariski triple with 12 lines whose fundamental groups of the complements are also not isomorphic.

In Appendix~\ref{App:examples}, we give twenty-eight new examples of arithmetic Zariski pairs of 11 lines with equations in a number field isomorphic to the $5$th cyclotomic field. They are all identified thanks to the loop linking number when they could not be detected by the $\I$-invariant.\\

We prove in Theorem~\ref{thm:multiplicativity}, that the loop linking number is multiplicative for the union of arrangements. This generalizes the multiplicativity theorem obtained by the author in~\cite{Gue:multiplicativity}. By applying this theorem, we re-prove that the complements of Rybnikov's arrangements are not homeomorphic. To our knowledge, this is the very first time this has been proven without computer assistance. Then, we generalize Rybnikov's construction answering to a weak version of the Falk-Randell Problem~1.2 (see~\cite{FalkRan}). Combining the multiplicativity theorem, the fact that the loop linking number is an invariant of the complement and the construction of homotopy-equivalent Zariski pairs made in~\cite{Gue:homotopy}, we deduce that there exist lattice-isomorphic arrangements with non-homeomorphic, albeit homotopy-equivalent, complements, see Theorem~\ref{thm:homotopy_vs_homeomorphism}.

\newpage
\subsection*{Notations}
\begin{itemize}
	\item $\Z(\A)=\bigcup_{L\in\A}L \subset \CC\PP^2$.
	\item $M(\A)=\CC\PP^2\setminus\Z(\A)$.
	\item $\A_P=\{L\in\A\mid P\in L\}$ is the arrangement formed by the lines of $\A$ containing~$P$.
	\item For $P\in\Z(\A)$, the multiplicity of $P$ is $m(P)=|\A_P|$.
	\item A point $P$ of $\Z(\A)$ is:
	\begin{itemize}
		\item smooth if $m(P)=1$,
		\item singular if $m(P)>1$,
		\item dense if $m(P)>2$.
	\end{itemize}
	\item $\Sing(\A)=\{P\in\Z(\A)\mid m(P)\geq2\}$ is the set of singular points of $\A$.
	\item A singular point $P$ such that $\A_P=\{L_{i_1},\cdots,L_{i_m}\}$ is denoted by $P_{i_1,\cdots,i_m}$ (or sometimes $\{i_1,\dots,i_m\}$).
	\item If $\A$ is a complex arrangement then $\overline{\A}$ is its complex conjugated arrangement.
\end{itemize}
	
\section{The loop linking number}\label{sec:construction}
	In this first section, we recall the construction of the loop linking number made by Cadegan-Schlieper in his Ph.D. thesis~\cite{Cad}. We refer to it for the details of the proofs.\\

Let $\A=\{L_1,\dots,L_n\}$ be a line arrangement of $\CC\PP^2$. In this paper, we assume that $\A$ is not a pencil\footnote{If $\A$ is a pencil then its topology is combinatorially determined.}. The arrangement $\A$ is \emph{ordered} if it comes with a total order. The \emph{combinatorics} of $\A$ is given by the set $\{\A_P \mid P\in\Sing(\A)\}$. The combinatorics admits a natural order inherited from the order of $\A$. The \emph{incidence graph} $\Gamma(\A)$ associated to $\A$ is the bipartite graph described as follows:
\begin{itemize} 
	\item the first set of vertices is composed of the \emph{point-vertices} $v_P$, for $P\in\Sing(\A)$,
	\item the second set of vertices is composed of the \emph{line-vertices} $v_L$), for $L\in\A$, 
	\item the edges of $\Gamma(\A)$ join~$v_P$ to~$v_L$ if and only if $P\in L$. 
\end{itemize}
In addition, we fix the orientation on the edges of $\Gamma(\A)$ from~$v_P$ to~$v_L$; thus we denote them~$(P\rightarrow L)$.\\

Let $G$ be an Abelian group. We consider the tensor space $\HH^1(M(\A);G)\otimes_\ZZ \HH_1(\Gamma(\A),\ZZ)$. From Orlik-Solomon~\cite{OS}, this space is determined by the combinatorics. Also, we consider the elements of $\HH^1(M(\A);G)$ as characters from $\HH_1(M(\A))\simeq\langle m_L \mid L\in\A \rangle/(m_{L_1}+\dots+m_{L_n})$ to~$G$. An element $\Lambda$ of this tensor space can be written as a chain
\begin{equation*}
	\Lambda = \sum_{P\in\Sing(\A)} \sum_{L\ni P} \lambda_\PtoL \otimes \PtoL,
\end{equation*}
with $\lambda_{\PtoL}\in\HH^1(M(\A);G)$; and which verifies the following boundary condition
\begin{equation}\label{eq:boundary}
	\partial \Lambda = \sum_{P\in\Sing(\A)} \sum_{L\ni P} \lambda_\PtoL \otimes (v_L - v_P) = 0.
\end{equation}

The \emph{tensor linking group} of $\A$, denoted by $\TLG(\A,G)$, is the subgroup of $\HH^1(M(\A);G)\otimes_\ZZ \HH_1(\Gamma(\A),\ZZ)$ formed by the elements which verify:
\begin{enumerate}[label=(\Roman*)]
	\item for all $\PtoL$, and all $L'\in\A$ containing $P$, we have $\lambda_\PtoL(m_{L'})=0_G$,
	\item for all $\PtoL$, and all $P'\in\Sing(\A)$ contained in $L$, we have $\sum_{L'\ni P'}\lambda_\PtoL(m_{L'})=0_G$.\\
\end{enumerate}

\begin{rmk}
	From Orlik-Solomon~\cite{OS}, the previous contruction can be done using only the combinatorics of~$\A$. So, if $C$ denotes this combinatorics, then we can define the tensor linking group of $C$, which is therefore denoted by $\TLG(C,G)$. 
\end{rmk}

The boundary manifold $B(\A)$ of $\A$ is the boundary of a regular tubular neighborhood of $\Z(\A)$. From Neumann~\cite{Neu} (see also~\cite{Wes}), $B(\A)$ is a graph manifold based on $\Gamma(\A)$. A \emph{coherent embedding} of $\Gamma(\A)$ in $B(\A)$ is an embedding as described in~\cite[Section~9]{Wal:Klasse_Mannigfaltigkeiten}. Basically, such an embedding sends the edge $\PtoL$ in the boundary of a tubular neighborhood of $L$. The embedding of $\Gamma(\A)$ in $B(\A)$ described in~\cite{FGM} is an example of coherent embedding. 

We denote by $j$ the map induced by a coherent embedding on the first homology group. Then, we consider $i$ the map induced by the inclusion of $B(\A)$ in $M(\A)$ on the first homology group, and we define the map $\Psi : \HH^1(M(\A);G)\otimes_\ZZ \HH_1(\Gamma(\A),\ZZ) \rightarrow \HH^1(M(\A);G)\otimes_\ZZ\HH_1(M(\A);\ZZ)$ by $\Psi=\Id_{\HH^1(M(A);G)}\otimes (i\circ j)$.
The natural pairing of $\HH^1(M(\A);G)\otimes_\ZZ\HH_1(M(\A);\ZZ)$ is denoted by $\pi$.

\begin{definition}
	The \emph{loop linking number} of $\A$ associated to $\Lambda\in\TLG(\A,G)$ is
	\begin{equation*}
		\L(\A,\Lambda) = \pi \circ \Psi (\Lambda) \in G.
	\end{equation*}
\end{definition}

It is well defined since the difference between two coherent embeddings vanishes when we take the pairing $\pi$ due to Conditions~(I) and~(II). For more details about this definition, we refer to~\cite[Section 3.3.1]{Cad}. Cadegan-Schlieper proves the following invariance theorem.

\begin{theorem}[Proposition 21 in~\cite{Cad}]\label{thm:invariance}
	Let $\A_1$ and $\A_2$ be two ordered line arrangements. If it exists a homeomorphism $h$ of $\CC\PP^2$ such that $h(\Z(\A_1))=\Z(\A_2)$, which preserves the orientation and the order of the $\A_i$'s, then for any $\Lambda \in \TLG(\A_1,G)$,
	\begin{equation*}
		\L(\A_1,\Lambda) = \L(\A_2,h_*(\Lambda)),
	\end{equation*}
	where $h_*:\TLG(\A_1,G)\rightarrow\TLG(\A_2,G)$ is the isomorphism induced by $h$.
\end{theorem}

In addition, he gives a method, related to homological algebra, to compute the loop linking number from the monodromy. He also derives from this description the following proposition, which allows to deduce the value of the loop linking number from the one of a Galois conjugated arrangement.

\begin{proposition}[Section~3.3.4 in~\cite{Cad}]\label{prop:Galois_invariance}
	Let $K$ be a Galois field over $\QQ$ containing the $N$th root of unity (for a fixed $N$). Let $\A$ be an arrangement defined by linear forms with coefficients in $K$. The actions of the Galois group $\Gal(K/\QQ)$ on the coefficients of the equation of $\A$ and on $\ZZ/N\ZZ$ commute with the loop linking number with values in $\ZZ/N\ZZ$. In other words, for all $\sigma\in\Gal(K/\QQ)$, we have
	\begin{equation*}
		\L(\sigma\cdot\A,\ZZ/N\ZZ) =\sigma\cdot \L(\A,\ZZ/N\ZZ).
	\end{equation*}
\end{proposition}

\section{Invariant of the complement \& derived invariants}
	
The invariance theorem obtained by Cadegan-Schlieper can be improved upon by showing that the loop linking number is an invariant of the homeomorphism type of $M(\A)$. Before, we need to clarify some definitions and constructions about the complement of arrangement.\\

Let $\hat{\CC\PP^2}$ be the image of the projective complex plane after blowing up the dense points of $\A$. We denote by $\hat{L}$ (resp. $E_P$) the strict transformation of $L$ (resp. the pre-image of $P$) in $\hat{\CC\PP^2}$, and let $\hat{\A}$ be the set $\{\hat{L},E_P \,\mid\, L\in\A,\, P\in\Sing(\A),\, m(P)>2\}$, and $\Z(\hat{\A})=\bigcup_{D\in\hat{\A}} D$. It is well-known that $M(\A)$ is homeomorphic to $\hat{\CC\PP^2}\setminus\Z(\hat{\A})$. 

\subsection{Automorphisms of the combinatorics}\mbox{}

An \emph{automorphism of the combinatorics} of $\A$ (resp. $\hat{\A}$) is a permutation of the element of $\A$ (resp. $\hat{\A}$) which respects the incidence relations. The set of all automorphisms of the combinatorics form a group, denoted by $\Aut(\A)$ (resp. $\Aut(\hat{\A})$). An arrangement is \emph{combinatorially stable} if $\Aut(\A)$ and $\Aut(\hat{\A})$ are isomorphic. For example, Ceva's arrangement is not combinatorially stable, while MacLane's arrangements are.

If $\A$ is an ordered arrangement, then $\hat{\A}$ inherits an order as follows: 
\begin{itemize}
	\item $\forall L,L'\in\A\, :\, \hat{L}<\hat{L}' \Leftrightarrow L<L'$,
	\item $\forall L\in\A,\, \forall P\in\Sing(\A),\, m(P)>2\, :\, \hat{L}<E_P$,
	\item $\forall P,P'\in\Sing(\A),\, m(P)>2,\, m(P')>2\, :$ 
		$$E_P<E_{P'} \Leftrightarrow \min( L\in\A \mid L\ni P, L\not\ni P') < \min( L\in\A \mid L\ni P', L\not\ni P),$$
		where the minimum is determined according to the fixed total order on the elements of $\A$.
\end{itemize}
Note that if $\A$ (resp. $\hat{A}$) is ordered, then only the identity of $\Aut(\A)$ (resp. of $\Aut(\hat{\A})$) respects this order.

\subsection{Topology of the complement}\label{sec:topology_complement}\mbox{}

Usually, the notion of a meridian of $L\in\A$ in $M(\A)$ is deduced from the embedded arrangement. More precisely, a meridian is the boundary of a small disk transverse to $L\setminus \Sing(\A)$. Unfortunately, a meridian is not preserved by homeomorphism of the complement. Indeed, it may exist a Cremona transformation of $\CC\PP^2$ which sends a line of $\A$ on a point, as it is the case in Ceva's arrangement (see~\cite[Remark~1.3]{ACC:Essential_coordinate}). Nevertheless, we will show that the notion of meridian is consistent as soon as we consider the meridians of the irreducible components of $\hat{\A}$. As we will see, such meridians can be constructed from the complement without the use of the embedding of $\hat{\A}$ in $\hat{\CC\PP^2}$.

From Jiang and Yau~\cite{JiaYau}, we know that the homeomorphism type of $M(\A)$ determines the combinatorics of $\A$. Then, Westlund proves in~\cite{Wes} (see also~\cite{Neu}) that the combinatorics of $\A$ determines the boundary manifold $B(\A)$ (but not his embedding in $M(\A)$).

A \emph{geometric filtration} of $M(\A)$ is an increasing filtration of $M(\A)$ by compact sets $C_1\subset \cdots \subset C_k \cdots $, such that there exists an $N\in\NN$ which verifies that for all $k>N$, $B(\A)$ is a deformation retract of $M(\A)\setminus C_k$. Note that such a filtration always exists. Indeed, let $\tub_k(\A)$ be a closed tubular neighborhood of $\Z(\A)$ such that for all $p\in\partial \tub_k(\A)$, $1/(k+1)<\dist(p,\Z(\A))<1/k$, the filtration $C_k=\CC\PP^2\setminus\stackrel{\circ}{\tub}_k(\A)$ is geometric since $M(\A)\setminus C_k=\stackrel{\circ}{\tub}_k(\A)\setminus\Z(\A)\simeq B(\A)\times(0,1)$.\\

\Warning The last homeomorphism between $\tub_k(\A)\setminus\Z(\A)$ and $B(\A)\times(0,1)$ is not unique. Nevertheless, a geometric filtration provides at least one embedding of $B(\A)$ in~$M(\A)$.
\begin{lemma}\label{lem:filtration}
	Geometric filtrations of $M(\A)$ are preserved by homeomorphisms from $M(\A)$ to $M(\A)$.
\end{lemma}

\begin{proof}
	Let $h$ be a homeomorphism of $M(\A)$, and let $C_1\subset \cdots \subset C_k \cdots $ be geometric filtration. It is clear that $h(C_1)\subset \cdots \subset h(C_k) \cdots $. Since $h$ is a homeomorphism of $M(\A)$ then $M(\A) \setminus h(C_k) \simeq h(M(\A)\setminus C_k) \simeq B(\A)\times(0,1)$. So $h(C_1)\subset \cdots \subset h(C_k) \cdots $ is a geometric filtration of $M(\A)$.
\end{proof}

The boundary manifold of $\hat{\A}$ is denoted by $B(\hat{\A})$ (which is homeomorphic to $B(\A)$). It is constructed as follows. For each component $D\in\hat{\A}$, we consider $T_D$ the trivial $S^1$-bundle over $\CC\PP^1\setminus\{d_1,\dots,d_r\}$, where $r$ is the cardinality of $\{D'\in\hat{\A}\setminus\{D\} \mid D\cap D'\neq\emptyset \}$, and the $d_i$'s are small disjoint discs. The boundary manifold $B(\hat{\A})$ is then obtained by gluing $T_D$ with $T_{D'}$ when $D\cap D'\neq\emptyset$, along the tori $S^1\times \partial d_i$ (identifying meridian with longitude). 

\begin{definition}
	A \emph{meridian} of $\hat{\A}$ is a path of $M(\A)$ homotopic to the image of a $S^1$-fiber of $T_D$ for a $D\in\hat{\A}$, by an inclusion of $B(\A)\simeq B(\hat{\A})$ in $M(\A)$ induced by a geometric filtration. 
\end{definition}

To assign a meridian $m$ to a component $D\in\hat{\A}$, we choose and fix an embedding $B(\A)\hookrightarrow M(\A)$, thus $m$ is a meridian of $D$ if and only if it is homotopically equivalent to a $S^1$-fiber of $T_D$ after embedding in~$M(\A)$.

\begin{lemma}\label{lem:meridians_preserving}
	A homeomorphism $h$ of $M(\A)$ preseves meridians and induces an automorphism of the combinatorics $\bar{h}\in\Aut(\hat{\A})$. 
\end{lemma}

\begin{proof}
	Using a geometric filtration, we choose and fix an embedding $B(\A)\hookrightarrow M(\A)$.	By Lemma~\ref{lem:filtration}, the homeomorphism $h$ induces an homotopy of $B(\A)$. Since $\A$ is not a pencil, then its fundamental group is sufficiently large. Using Waldhausen~\cite[Corollary~6.5]{Wal:Irreducible_3-manifolds}, this homotopy induces a homeomorphism $h^B$ of $B(\A)$. From another result of Waldhausen~\cite[Satz~(10.1)]{Wal:Klasse_Mannigfaltigkeiten}, this homeomorphism $h^B$ respects the graph strucutre of $B(\hat{\A})$ and the $S^1$-fibers of the $T_D$'s. The first part implies that it exists an automorphism $\bar{h}\in\Aut(\hat{\A})$ such that $h^B(T_D)=T_{\bar{h}(D)}$; while the second part, together with the definition of meridian, implies that $h$ respects meridians.
\end{proof}

\begin{definition}\label{def:ordered_topology}
	The \emph{ordered topology} of $M(\A)$ is the class of homeomorphisms of $M(\A)$ which induce the identity in $\Aut(\hat{\A})$.
\end{definition}

To fix an order on the topology of $M(\A)$ is to choose and fix an embedding $e:B(\A)\hookrightarrow M(\A)$ associate to a geometric filtration. Thus, a homeomorphism $h$ of $M(\A)$ respects the ordered topology if and only if for $D\in\hat{\A}$, $h(m_D)\simeq m_D$. Note that this is independent of the choice of the embedding $B(\A)\hookrightarrow M(\A)$. 

When it comes to compare the ordered topology of two arrangement complements, things are a bit different. Let $\A_1$ and $\A_2$ be two ordered arrangements with isomorphic\footnote{This isomorphism between the ordered combinatorics is unique, therefore it is oftenlty refered as the identity.} ordered combinatorics. This isomorphism induces a homeomorphism $\tilde{\sigma}:B(\A_1)\rightarrow B(\A_2)$. We fix an order on the topology of $M(\A_i)$ by choosing embeddings $e_i:B(\A_i)\hookrightarrow M(\A_i)$. A homeomorphism $h:M(\A_1)\rightarrow M(\A_2)$ respects the ordered topologies (given by the $e_i$'s) if the map $h^B:B(\A_1)\rightarrow B(\A_2)$ induced by $h$ is homotopic to $\tilde{\sigma}$.

\begin{rmk}
	When we consider the ordered topology of the pair $(\CC\PP^2,\A)$, the embedding of $B(\A)$ in $M(\A)$ is naturally fixed by the embedding of $\A$ in $\CC\PP^2$.
\end{rmk}

\subsection{Invariant of the complement}\mbox{}

As previously mentioned, the invariance result of Cadegan-Schlieper (Theorem~\ref{thm:invariance}) can be improved upon by proving that the loop linking number is an invariant of the complement. The proof of this theorem is similar to the one of Theorem~2.2 in~\cite{ACM:triangular}.

\begin{theorem}\label{thm:invariance_complement}
	Let $\A_1$ and $\A_2$ be two combinatorially stable arrangements. If it exists a homeomorphism $h$ from $M(\A_1)$ to $M(\A_2)$ which respects the order and the orientation, then 
	\begin{equation*}
		\L(\A_1,\Lambda)=\L(\A_2,\tilde{h}(\Lambda)),
	\end{equation*}
	where $\tilde{h}:\TLG(\A_1,G)\rightarrow\TLG(\A_2,G)$ is the isomorphism induced by $h$.
\end{theorem}

\begin{proof}
	If the $\A_i$'s are combinatorially stable then the homeomorphism $\bar{h}$ induced by $h$ on the intersection graphs of the $\hat{\A}_i$'s is, in fact, an isomorphism between the intersection graphs of the $\A_i$'s. Since $\TLG(\A_i)$ is combinatorially determined, then $\bar{h}$ induces an isomorphism $\tilde{h}:\TLG(\A_1)\rightarrow\TLG(\A_2)$.
	
	By assumption, $M(\A_i)$ is ordered, which means that we have fixed the embeedings $B(\A_i)\hookrightarrow M(\A_i)$ (see Section~\ref{sec:topology_complement}). The homeomorphism $h$ induces a homeomorphism $h^B:B(\A_1)\rightarrow B(\A_2)$ which respects the graph structures (by~\cite{Wal:Irreducible_3-manifolds,Wal:Klasse_Mannigfaltigkeiten}); thus, it sends a coherent embedding of $\Gamma(\A_1)$ on a coherent embedding of $\Gamma(\A_2)$. This implies that we have the following commutative diagram.
	\begin{center}
		\begin{tikzcd}
			\HH_1(\Gamma(\A_1)) \arrow[r] \arrow[d,"\Id"] & \HH_1(\B(\A_1)) \arrow[r,"{e_1}_*"] \arrow[d,"h^B_*"] & \HH_1(M(\A_1)) \arrow[d,"h_*"] \\
			\HH_1(\Gamma(\A_2)) \arrow[r] 					& \HH_1(\B(\A_2)) \arrow[r,"{e_2}_*"] 					& \HH_1(M(\A_2))
		\end{tikzcd}
	\end{center}
	The theorem then arises from the definition of the loop linking number.
\end{proof}

\subsection{Derived invariants}\mbox{}

To remove the "ordered and oriented" condition in the previous theorem, we define the \emph{full loop linking number} of $\A$ associated to $\Lambda\in\TLG(\A)$ as
\begin{equation*}
	\Lb(\A,\Lambda)=\big\{\, \L(\sigma\cdot\A,\Lambda)^{\pm 1}\, \mid\, \sigma\in\Aut(\A)\, \big\}\subset G.
\end{equation*}

\begin{theorem}\label{thm:full_loop_linking}
	Let $\A_1$ and $\A_2$ be two combinatorially stable line arrangements. If it exists a homeomorphism $h$ from $M(\A_1)$ to $M(\A_2)$, then for any $\Lambda \in \TLG(\A_1,G)$,
	\begin{equation*}
		\Lb(\A_1,\Lambda) = \Lb(\A_2,\tilde{h}(\Lambda)),
	\end{equation*}
	where $\tilde{h}:\TLG(\A_1,G)\rightarrow\TLG(\A_2,G)$ is the isomorphism induced by $h$.
\end{theorem}

\begin{proof}
	First, if $h$ does not respect the global orientation, then its composition with the complex conjugation respects the global orientation. We deduce from Proposition~\ref{prop:Galois_invariance} that $\L(\A_1,\Lambda)=\L(\A_2,\tilde{h}(\Lambda))^{-1}$. Second, if $h$ does not respect the local orientation of the meridians, then $h$ reverses the orientation of all the meridians (by the connectivity of $\hat{\A}$, see~\cite{ACCM}). It follows that $\L(\A_1,\Lambda)=\L(\A_2,\tilde{h}(\Lambda))^{-1}$.
	
	By~Lemma~\ref{lem:meridians_preserving}, if a homeomorphism $h:M(\A_1)\to M(\A_2)$ exists then it induces an isomorphism $\bar{h}$ between the combinatorics of $\A_1$ and $\A_2$. This implies that $h$ is a homeomorphism from $M(\bar{h}\cdot\A_1)$ to $M(\A_2)$ which respects the orders. By Theorem~\ref{thm:invariance_complement}, we get
	\begin{equation*}
		\L(\bar{h}\cdot\A_1,\Lambda)=\L(\A_2,\tilde{h}(\Lambda))^{\pm 1}.
	\end{equation*}
	We conclude noticing that $\Lb(\bar{h}\cdot\A_1,\Lambda) = \Lb(\A_1,\Lambda)$.
\end{proof}

\begin{corollary}\label{cor:full_loop_linking}
	Let $\A_1,\A_2$ be two combinatorially stable line arrangements and let $\Lambda \in \TLG(\A_1,G)$ be such that for any $\sigma\in\Aut(\A_1)$, we have $\L(\sigma\cdot\A_1,\Lambda)=\L(\A_1,\Lambda)$. If it exists a homeomorphism $h:M(\A_1)\rightarrow M(\A_2)$, then
	\begin{equation*}
		\L(\A_1,\Lambda) = \L(\A_2,\tilde{h}(\Lambda))^{\pm 1},
	\end{equation*}
	where $\tilde{h}:\TLG(\A_1,G)\rightarrow\TLG(\A_2,G)$ is the isomorphism induced by $h$.
\end{corollary}

\begin{proof}
	If the loop linking number is invariant under the action of $\Aut(\A)$ then $\Lb(\A,\Lambda)=\{\L(\A,\Lambda)^{\pm 1}\}$. The conclusion follows from Theorem~\ref{thm:full_loop_linking}.
\end{proof}

\begin{rmk}\label{rmk:trivial_automorphism}
	If $\Aut(\A)$ is trivial, then $\{\L(\A,\Lambda)^{\pm 1}\}$ is a topological invariant of $M(\A)$.
\end{rmk}

\section{Topological computation}\label{sec:computation}
	
Throughout this section, we describe the loop linking number from a geometrical point of view. For the algebraic one, we refer to~\cite{Cad}. In this section, the singular points of $\A$ are denoted by $P_1,\dots,P_m$.\\

Let $P_0$ be a point of $M(\A)$ and $F_0$ a line passing through $P_0$ generic with $\Z(\A)$ (i.e. $\# F_0\cap\Z(\A) = n$). Let $q:\CC\PP^2\setminus P_0 \rightarrow \CC\PP^1$ be the natural projection defined by $P_0$. For each edge $\PtoL$ of $\Gamma_\A$, we define the geometric braid $B_\PtoL$ with $n-m(P)+1$ strands as follows. Let $R_P$ be a smooth path (without self-intersection) in $\CC\PP^1$ from $q_0:=q(F_0\setminus P_0)$ to $q(P)$ and such that $q(\Sing(\A)) \cap R_P= q(P)$. This last condition implies that the $\CC$-fiber over any point of $R_P$ intersects $\A_\PtoL=\{L\}\cup (\A\setminus\A_P)$ in exactly $n-m(P)+1$ points. So, the geometric braid $B_\PtoL$ is defined as the intersection $\Z(\A_\PtoL)\cap q\inv(R_P)\subset R_P\times\CC$.

\begin{rmk}\label{rmk:braid_monodromy}
	If a braid monodromy of $\A$ based in $P_0$ is given by $(b_1 t_1 b_1\inv,\cdots,b_m t_m b_m\inv)$, where $t_i$ is the local full-twist associated to the singular point $P_i\in\Sing(\A)$ (see~\cite{Moi,CohSuc:braid_monodromy} for details), then $B_{P_i\rightarrow L}$ is the sub-braid of $b_i$ obtained by removing the strands associated to the lines of $\A_{P_i}\setminus{L}$.
\end{rmk}

Let $\Re:R_P \times \CC \rightarrow R_P\times\RR$ be the projection on the real part of the term $\CC$ in $R_P\times\CC$. Up to a sligth perturbation, we can assume that $\Re(B_{\PtoL})$ has only double points. So, we denote by $D_\PtoL=\sigma_1^{\epsilon_1}\cdots\sigma_k^{\epsilon_k}$, with $\epsilon_i=\pm 1$, the diagram of $B_\PtoL$ associated to $\Re$. Let $\ulk_L(B_\PtoL)$ be the \emph{upper-linking} of $L$ with $B_\PtoL$ defined by
\begin{equation*}
	\ulk_L(B_\PtoL) = \sum_{i=1}^k \epsilon_i.\delta_L(\sigma_i^{\epsilon_i}),
\end{equation*}
where $\delta_L(\sigma_i^{\epsilon_i})$ is the meridian of the strand over-crossing in $\sigma_i^{\epsilon_i}$ if the under-crossing strand is $L$, otherwise it is $0$.

\begin{theorem}\label{thm:computation}
	The loop linking number of $\A$ associated to $\Lambda\in\TLG(\A,G)$ is given by
	\begin{equation*}
		\L(\A,\Lambda) = \sum_{P\in\Sing(\A)} \sum_{L\ni P} \lambda_\PtoL \big(\ulk_L(B_\PtoL)\big).
	\end{equation*}
\end{theorem}

\begin{proof}
	Let $F_{P_i}$ (sometimes simply $F_i$ to lighten the notation) be the unique line of the pencil centered in $P_0$ which contains the singular point $P_i\in\Sing(\A)$. We consider the fibered arrangement $\widetilde{\A}$ defined by $\widetilde{\A}=\A\cup\big(\bigcup_{i=0}^m F_i\big)$. For each $P\in\Sing(\A)$, we set $e_P$ a path in $F_P\setminus\Z(\A\setminus\A_P)$ joining $P_0$ to $P$; similarly, for each $L\in\A$, we set $e_L$ a path of $F_0\setminus\Z(\A\setminus \{L\})$ joining $L\cap F_0$ to $P_0$. In~\cite{FGM}, the coherent embedding $\Gamma(\A)\rightarrow B(\A)$ is described as the composition of a map $j_0$ from $\Gamma(\A)$ to $\Z(\A)$ with a map pushing $j_0(\Gamma(\A))$ in $B(\A)$. We can assume that the first inclusion $j_0$ sends the line-vertices $v_L$ on $L\cap F_0$. We denote by $\E_{\PtoL}$ the cycle of $M(\A)$ defined by $\tilde{e}_L + i\circ j \big( \PtoL \big) + \tilde{e}_P$, where $\tilde{e}_P$ (resp. $\tilde{e}_L$) is the image of $e_P$ (resp. $e_L$) by the push of $j_0(\Gamma(\tilde{\A})$ in $B(\tilde{\A})$ composed with the inclusion of $B(\tilde{\A})$ in $M(\tilde{\A})\simeq M(\A)$. By definition, the loop linking number is given by
	\begin{equation*}
		\L(\A,\Lambda) = \sum_{P\in\Sing(\A)} \sum_{L\ni P} \pi\big( \lambda_{\PtoL}\otimes i\circ j\big(\PtoL\big) \big).
	\end{equation*}
	Due to the boundary condition, we have
	\begin{align*}
		\L(\A,\Lambda) 
			& = \sum_{P\in\Sing(\A)} \sum_{L\ni P} \lambda_{\PtoL}\otimes \big(\tilde{e}_L + i\circ j \big( \PtoL \big) + \tilde{e}_P\big),\\
			& = \sum_{P\in\Sing(\A)} \sum_{L\ni P} \lambda_{\PtoL}(\E_{\PtoL}).
	\end{align*}
	According to~\cite[Lemma~4.3]{AFG} which describes the map $j:\HH_1(B(\tilde{A}))\rightarrow\HH_1(M(\tilde{\A}))$ (see also \cite[Theorem~4.3]{FGM}) for cycles supported by three lines (here $F_0$, $L$ and $F_P$), we obtain
	\begin{equation*}
		\L(\A,\Lambda) = \sum_{P\in\Sing(\A)} \sum_{L\ni P} \lambda_\PtoL \big(\ulk_L(B_{\PtoL})\big).
		\qedhere
	\end{equation*}
\end{proof}

\begin{rmk}\label{rmk:computation_automorphism}
	Using the formula obtained in Theorem~\ref{thm:computation}, we can easily compute the value of $\L(\sigma\cdot\A,\Lambda)$, for any $\sigma\in\Aut(\A)$, since we have:
	\begin{equation*}
		\L(\sigma\cdot\A,\Lambda) = \sum_{P\in\Sing(\A)} \sum_{L\ni P} \lambda_\PtoL \big(\ulk_{\sigma\cdot L}(B_{\sigma\cdot\PtoL)}\big),
	\end{equation*}
	where $\sigma\cdot\PtoL$ is the edge of $\Gamma_\A$ which joins the point-vertex $v_{\sigma\cdot P}=v_P$ to the line-vertex $v_{\sigma\cdot L}$.
\end{rmk}

In the following subsections, we give two effective methods to compute $\ulk_{L}(B_{\PtoL})$. The first uses the notion of braided wiring diagram, while the second is more algorithmic and uses changes of variables to compute the different upper-linkings.

\subsection{Computation via the braided wiring diagram}\label{sec:computation_wiring}\mbox{}

Roughly speaking, the \emph{braided wiring diagram} (or shortly the \emph{wiring diagram}) is the trace of the arrangement $\A$ in the fibers over a smooth path $\rho:[0,1]\rightarrow\CC\PP^1$ starting from $q_0$ and passing through all the points of $q(\Sing(\A))$ (see~\cite{Arvola,CohSuc:braid_monodromy} for details). It is a singular braid, whose singular points correspond to the singular points of $\A$. \\

We order the points $\{P_1,\dots,P_m\}$ of $\Sing(\A)$ according to the order of their image in $\rho$, and we re-parameter $\rho$ such that $\rho(i/m)=P_i$, for $i\in\{0,\cdots,m\}$. A wiring diagram $\W(\A)$ of $\A$ can be given as an ordered $|\Sing(\A)|$-tuple of pairs formed by a braid $b_i\in\BB_n$ and a singular point $P_i\in\Sing(\A)$:
\begin{equation*}
	\W(\A)=\big[\, 
	[b_1, P_1],
	\cdots,
	[b_m, P_m]
	\,\big],
\end{equation*}
where $b_i=\A\cap q\inv\big( \widehat{P_{i-1} P_{i}} \big)\in\BB_n$, with $\widehat{P_{i-1} P_{i}}=\rho\big((\frac{i-1}{m},\frac{i}{m})\big)$.

For a fixed $i\in\{1,\cdots,m\}$, the braid $B_{(P_i\to L)}$ can be obtained from $\W(\A)$ as follows. Consider that the path $R_{P_i}$ defined in the previous section is a slight deformation of $\rho\big((0,\frac{i}{m})\big)$ which avoids the points $q(P_1),\dots,q(P_{i-1})$ turning around them counter-clockwise. In such a situation, we define:
\begin{equation*}
	\overline{B}_{\PtoL}=b_1\cdot T_1\cdot \ldots \cdot T_{i-1} \cdot b_i,
\end{equation*}
where $T_j$ is the local positive half-twist of the strands which correspond to the lines of $\A_{P_j}$ (see~\cite[Section~4]{AFG} for an example). The braid $B_{(P_i\to L)}$ is obtained from $\overline{B}_{(P_i\to L)}$ by removing the strands which correspond to the lines $\A_{P_i}$, except $L$.

The computation is completed using Theorem~\ref{thm:computation}.

\subsection{Computation using changes of variables}\label{sec:computation_CoV}\mbox{}

The objective of this method is to apply a linear change of variable $\Delta_{\PtoL}$ on the arrangement $\A$, in order to have a simple computation of $\ulk_L(B_{\PtoL})$.

For an edge $\PtoL$ in $\Gamma_\A$, let $\Delta_{\PtoL}$ be a linear change of variable which sends $F_0$, $F_P$ and $L$ on $x=0$, $x-z=0$ and $y=0$ respectively. We take as path $R_P$ the image by $q$ of the segment which joins $[1:0:0]$ to $[1:0:1]$. The trace of $\Z(\Delta_{\PtoL}(\A))$ in the fibers over $R_P$ is a union of $n$ segments $S_{L_1},\dots,S_{L_n}$ in $R_P\times\CC$ (each one corresponding with a line of $\Delta_{\PtoL}(\A)$). Notice that after the projection $\Re:R_P \times \CC \rightarrow R_P\times\RR$, two segments intersect once at most. 

In order to simplify the computation, we work in the chart $z=1$ and we consider $R_P\times \CC$ as $[0,1]\times\RR^2$, where the first coordinate in $\RR^2$ corresponds to the real part of $\CC$ and the second to the imaginary part.\\

Let $L'$ be a line of $\A\setminus\A_P$. We denote by $(0,q_1,q_2)$ (resp. $(1,p_1,p_2)$) the coordinates in $[0,1]\times\RR^2$ of the intersection point of the line $\Delta_{\PtoL}(L')$ with $x=0$ (resp. $x-1=0$). The segment $S_{L'}$, corresponding in $[0,1]\times\RR^2$ to the line $L'$, is the segment defined by:
\begin{equation*}
	S_{L'}=\big(\ t,\ q_1 (1-t) + p_1 t,\	q_2 (1-t) + p_2 t\ \big), \text{ for }t\in [0,1].
\end{equation*}
Take $\mathbf{t}=q_1/(p_1+q_1)$. We define the function $\Phi_{\PtoL}(L')$ as follows: $\Phi_{\PtoL}(L')=0$ except if $0<\mathbf{t}<1$ and
\begin{itemize}
	\item $q_2 (1-\mathbf{t}) + p_2 \mathbf{t} = 0$ and $q_1>0$, then $\Phi_{\PtoL}(L')=1$;
	\item $q_2 (1-\mathbf{t}) + p_2 \mathbf{t} > 0$ and 
	\begin{itemize}
		\item $q_1>0$, then $\Phi_{\PtoL}(L')=1$;
		\item $q_1<0$, then $\Phi_{\PtoL}(L')=-1$.
	\end{itemize}
\end{itemize}

\begin{rmk}
	The first case in the construction of $\Phi_{\PtoL}$ corresponds to an actual crossing in $R_P\times\RR^2$ between $S_L$ and $S_{L'}$, while the second case corresponds to a virtual crossing. The two sub-cases describe the sign of this virtual crossing.
\end{rmk}

Using the computations made in~\cite[Section~4]{AFG}, we obtain
\begin{equation*}
	\ulk_L(B_\PtoL) = \sum_{L'\in\A\setminus\A_P} \Phi_{\PtoL}(L') \cdot m_{L'},
\end{equation*}
where $m_{L'}$ is a meridian of $L'$. As previously, we conclude using  Theorem~\ref{thm:computation}.

\section{Applications}\label{sec:application}
	\subsection{Ordered Zariski pair with ten lines}\label{sec:oZP}\mbox{}

Let $\zeta$ be a fixed root of the 5th cyclotomic polynomial $Z^4+Z^3+Z^2+Z+1$.  We consider the arrangements $\M_i$, for $i\in\{1,2,3,4\}$, defined by the following equations, where $\alpha=\zeta^i$.

\begin{equation*}
	\begin{array}{l p{0.5cm} l }
		L_1: z=0 &&
		L_2: x =0 \\
		L_3: (\alpha^2+\alpha)x -(\alpha+1)y + z=0 && 
		L_4: y=0\\ 
		L_5: y-z = 0 && 
		L_6: x-z=0\\ 
		L_7: \alpha x -(\alpha+1) y+z=0 && 
		L_8: (\alpha^2+\alpha)x + -(\alpha^2+\alpha+1)y+ z = 0\\ 
		L_9: \alpha x -y + z =0 &&
		L_{10}: \alpha^3 x + y =0
	\end{array}
\end{equation*}

The combinatorics\footnote{To lighten the notation a sub-arrangement $\A_P=\{L_{i_1},\cdots,L_{i_m}\}$ is denoted by $\{i_1,\cdots,i_m\}$.} $C$ shared by the $\M_i$'s is:
\begin{align*}
	C = \big\{\, 
		& \{1, 2, 6\}, \{1, 3, 9\}, \{1, 4, 5\}, \{1, 7\}, \{1, 8, 10\}, \{2, 3, 7\}, \{2, 4, 10\}, \{2, 5, 9\}, \{2, 8\}, \\
		& \{3, 4, 8\}, \{3, 5\}, \{3, 6, 10\}, \{4, 6\}, \{4, 7, 9\}, \{5, 6, 7, 8\}, \{5, 10\}, \{6, 9\}, \{7, 10\}, \{8, 9\}, \{9, 10\}\, \big\}.
\end{align*}

\begin{lemma}
	The automorphsim group of $C$ is the subgroup of order $4$ in the symetric group $\Sigma_{10}$ generated by the permutation:
	\begin{equation*}
		\sigma=(1,2,3,4)(5,6,7,8)(9,10).
	\end{equation*}
	Its action on the lines of $C$ is given by $\sigma\cdot L_i = L_{\sigma(i)}$.
\end{lemma}

\begin{proof}
	First, we notice that:
	\begin{itemize}
		\item[-] $L_1,L_2,L_3,L_4$ contain each 4 triple points and 1 double point;
		\item[-]	$L_5,L_6,L_7,L_8$ contain each 1 quadruple point, 2 triple points and 2 double points;
		\item[-] $L_9,L_{10}$ contain each 3 triple points and 3 double points.
	\end{itemize}
	This implies that for any $\tau\in\Aut(C)$, we have $\tau(\{L_1,L_2,L_3,L_4\})=\{L_1,L_2,L_3,L_4\}$, $\tau(\{L_5,L_6,L_7,L_8\})=\{L_5,L_6,L_7,L_8\}$ and $\tau(\{L_9,L_{10}\})=\{L_9,L_{10}\}$.
	We conclude with a case by case study of the implications of $\tau(L_1)=L_j$, for $j\in\{1,2,3,4\}$.
	
	If $\tau(L_1)=L_2$, then:
	\begin{itemize}
		\item[-] the double point $\{1,7\}$ is sent on $\{2,8\}$, so $\tau(L_7)=L_8$;
		\item[-] the double point $\{8,9\}$ is sent on $\{7,10\}$, so $\tau(L_9)=L_{10}$ and $\tau(L_{10})=L_9$;
		\item[-] the triple point $\{1,8,10\}$ (resp. $\{1,3,9\}$) is sent on $\{2,5,9\}$ (resp. $\{2,4,10\}$), so $\tau(L_8)=L_5$ and $\tau(L_3)=L_4$;
		\item[-] the double point $\{3,5\}$ is sent on $\{4,6\}$, so $\tau(L_5)=L_6$;
		\item[-] the triple $\{2,3,7\}$ (resp. $\{1,4,5\}$ and $\{3,6,10\}$) is sent on $\{3,4,8\}$ (resp. $\{2,1,6\}$ and $\{4,7,9\}$), so $\tau(L_2)=L_3$, $\tau(L_4)=L_1$ and $\tau(L_6)=L_7$;
	\end{itemize}
	and thus $\tau=\sigma$. Similarly, if $\tau(L_1)=L_1$, $\tau(L_1)=L_3$ and $\tau(L_1)=L_4$ respectively, then $\tau=\Id$, $\tau=\sigma^2$ and $\tau=\sigma^3$ respectively. 
\end{proof}

\begin{lemma}\label{lem:TLG}
	The tensor linking group $\TLG(C,\ZZ/5\ZZ)$ is isomorphic to $\ZZ/5\ZZ$, and it is generated by the tensor $\Lambda_0$ given in Appendix~\ref{App:values}.
\end{lemma}

\begin{proof}
	The tensor linking group can be viewed as the right kernel of the matrix with $\ZZ/5\ZZ$ coefficients given by the linear forms associated to the equations~(\ref{eq:boundary}) and those of Conditions~(I) and~(II) of the definition of the tensor linking group (see Section~\ref{sec:construction}). A simple computation shows that this kernel is generated by $\Lambda_0$, then we obtain the lemma.
\end{proof}

To compute the loop linking number $\L(\M_i,\Lambda)$, we use here the method described in Section~\ref{sec:computation_wiring}. To minimize the size of the braids $b_i$ in the wiring diagrams, we consider $P_0$ in a small neighborhood of $P_{5,6,7,8}$ and $F_0$ in a tubular neighborhood of $L_8$ (this makes the lines $L_5$, $L_6$, $L_7$ and $L_8$ "almost vertical"). Using the notation\footnote{The braids $B_{\PtoL}$ are given as tuple of integer $(i_1,\cdots,i_k)$, where a positive integer $i$ in the tuple indicates a positive crossing between the strands number $i$ and $|i+1|$; while a negative $i$ indicates a negative crossing between strands numbers $|i|$ and $|i|+1$. The singular points $P_{i_1,\cdots, i_m}$ are expressed as the tuple $[i_1,\dots,i_m]$.} described in Section~\ref{sec:computation_wiring}, we have the following braided wiring diagrams\footnote{A wiring diagram of $\M_2$ (resp. $\M_4$) can be obtained from the one of $\M_3$ (resp. $\M_1$) by the inversion of the sign of the virtual crossings.} for $\M_1$ and $\M_3$, they are also pictured in Figure~\ref{fig:wiring_M1} and~\ref{fig:wiring_M3}.

\begin{figure}[h!]
	\begin{tikzpicture}
	\begin{scope}[xscale=0.25,yscale=-0.4]
\node at (-10,1) {$L_{8}$};
\node at (-10,2) {$L_{6}$};
\node at (-10,3) {$L_{7}$};
\node at (-10,4) {$L_{5}$};

\node at (-10,5) {$L_{2}$};

\node at (-10,6) {$L_{1}$};
\node at (-10,7) {$L_{10}$};

\node at (-10,8) {$L_{9}$};

\node at (-10,9) {$L_{4}$};
\node at (-10,10) {$L_{3}$};

\nocrossing{10}{-9}

\actualcrossing{10}{-8}{1}{4}
\actualcrossing{10}{-5}{4}{2}
\actualcrossing{10}{-4}{5}{3}
\actualcrossing{10}{-2}{7}{2}
\actualcrossing{10}{-1}{8}{3}

\actualcrossing{10}{1}{6}{3}
\undercrossing{10}{3}{7}
\overcrossing{10}{4}{8}
\overcrossing{10}{5}{4}
\overcrossing{10}{6}{5}
\actualcrossing{10}{7}{3}{3}
\actualcrossing{10}{9}{5}{3}
\actualcrossing{10}{11}{7}{2}
\actualcrossing{10}{12}{8}{2}
\overcrossing{10}{13}{5}
\undercrossing{10}{14}{4}
\undercrossing{10}{15}{3}
\undercrossing{10}{16}{7}
\actualcrossing{10}{17}{2}{3}
\actualcrossing{10}{19}{4}{2}
\actualcrossing{10}{20}{5}{2}
\actualcrossing{10}{21}{6}{3}
\overcrossing{10}{23}{6}
\overcrossing{10}{24}{2}
\overcrossing{10}{25}{5}
\actualcrossing{10}{26}{4}{2}
\overcrossing{10}{27}{3}
\actualcrossing{10}{28}{1}{3}
\actualcrossing{10}{30}{3}{2}
\actualcrossing{10}{31}{4}{2}
\actualcrossing{10}{32}{5}{3}
\undercrossing{10}{34}{1}
\undercrossing{10}{35}{2}
\undercrossing{10}{36}{1}
\undercrossing{10}{37}{3}
\undercrossing{10}{38}{4}
\undercrossing{10}{39}{5}
\actualcrossing{10}{40}{2}{3}
\nocrossing{10}{42}
	\end{scope}
\end{tikzpicture}
	\vspace{-0.5cm}
	\caption{Braided wiring diagram of $\M_1$.\label{fig:wiring_M1}}
\end{figure}
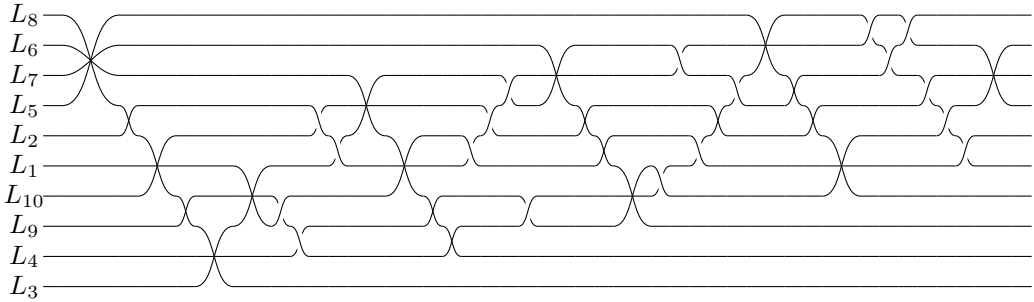

\begin{figure}[h!]
	\begin{tikzpicture}
	\begin{scope}[xscale=0.25,yscale=-0.4]
	
\node at (-10,1) {$L_{8}$};
\node at (-10,2) {$L_{7}$};
\node at (-10,3) {$L_{5}$};
\node at (-10,4) {$L_{6}$};

\node at (-10,5) {$L_{1}$};
\node at (-10,6) {$L_{10}$};

\node at (-10,7) {$L_{9}$};

\node at (-10,8) {$L_{4}$};
\node at (-10,9) {$L_{3}$};

\node at (-10,10) {$L_{2}$};

\nocrossing{10}{-9}

\actualcrossing{10}{-8}{1}{4}
\actualcrossing{10}{-5}{4}{3}
\actualcrossing{10}{-3}{6}{2}
\actualcrossing{10}{-2}{7}{3}
\actualcrossing{10}{0}{9}{2}

\actualcrossing{10}{1}{5}{3}
\undercrossing{10}{3}{6}
\undercrossing{10}{4}{5}
\overcrossing{10}{5}{8}
\overcrossing{10}{6}{7}
\actualcrossing{10}{7}{3}{2}
\actualcrossing{10}{8}{4}{2}
\actualcrossing{10}{9}{5}{3}
\actualcrossing{10}{11}{7}{3}
\overcrossing{10}{13}{5}
\undercrossing{10}{14}{6}
\overcrossing{10}{15}{4}
\undercrossing{10}{16}{7}
\actualcrossing{10}{17}{2}{2}
\actualcrossing{10}{18}{3}{2}
\actualcrossing{10}{19}{4}{3}
\actualcrossing{10}{21}{6}{3}
\overcrossing{10}{23}{6}
\overcrossing{10}{24}{5}
\actualcrossing{10}{25}{1}{3}
\actualcrossing{10}{27}{3}{2}
\actualcrossing{10}{28}{4}{2}
\actualcrossing{10}{29}{5}{3}
\undercrossing{10}{31}{2}
\actualcrossing{10}{32}{3}{2}
\undercrossing{10}{33}{3}
\undercrossing{10}{34}{4}
\overcrossing{10}{35}{2}
\undercrossing{10}{36}{5}
\overcrossing{10}{37}{5}
\actualcrossing{10}{38}{2}{3}
\nocrossing{10}{40}
	\end{scope}
\end{tikzpicture}
	\vspace{-0.5cm}
	\caption{Braided wiring diagram of $\M_3$.\label{fig:wiring_M3}}
\end{figure}
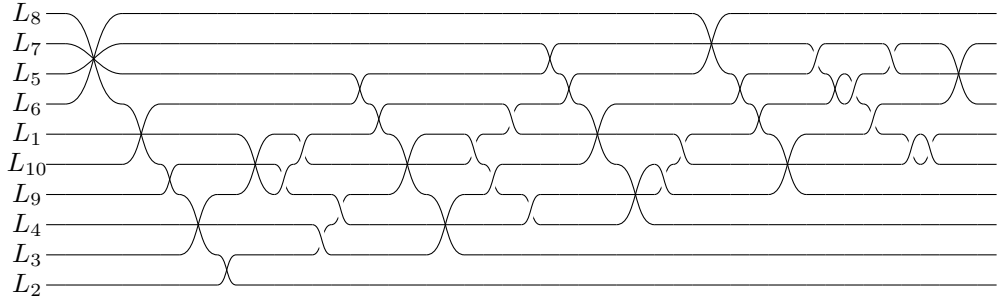

\begin{align*}
	\W(\M_1) = \big[\,
	& [(),[8,6,7,5]], [(),[8,2]], [(),[8,1,10]], [(),[8,9]], [(),[8,4,3]],\\ 
	& [(), [1, 9, 3]], [(-7, 8, 4, 5), [6, 10, 3]], [(), [6, 2, 1]], [(), [6, 4]], [(), [6, 9]], \\
	& [(5, -4, -3, -7), [7, 2, 3]], [(), [7, 10]], [(), [7, 1]], [(), [7, 9, 4]], [(6, 2, 5), [10, 9]], \\
	& [(3), [5, 2, 9]], [(), [5, 3]], [(), [5, 10]], [(), [5, 1, 4]], [(-1, -2, -1, -3, -4, -5), [2, 10, 4]]\, \big],
\end{align*}

\begin{align*}
	\W(\M_3) = \big[\, 
	& [(),[8,7,5,6]], [(),[8,1,10]], [(),[8,9]], [(),[8,4,3]], [(),[8,2]], \\
	& [(), [1, 9, 3]], [(-6, -5, 8, 7), [7, 10]], [(), [7, 1]], [(), [7, 3, 2]], [(), [7, 9, 4]], \\
	& [(5, -6, 4, -7), [5, 10]], [(), [5, 3]], [(), [5, 1, 4]], [(), [5, 9, 2]], [(6, 5), [6, 10, 3]], \\
	& [(), [6, 4]], [(), [6, 9]], [(), [6, 1, 2]], [(-2), [10, 9]], [(-3, -4, 2, -5, 5), [10, 4, 2]]\, \big].
\end{align*}

\begin{theorem}\label{thm:ooZP}
	For any non-trivial $\Lambda\in\TLG(C,\ZZ/5\ZZ)$, we have
	\begin{equation*}
		\L(\M_i,\Lambda) \neq 0.
	\end{equation*}
\end{theorem}

\begin{proof}
	By Proposition~\ref{prop:Galois_invariance}, it is enough to prove the result only for $\M_1$. By Lemma~\ref{lem:TLG}, we can restrict to the non-trivial element $\Lambda_0$ of $\TLG(C,\ZZ/5\ZZ)$. To compute the loop linking number, we apply the method described in Section~\ref{sec:computation_wiring} using the values given in Appendix~\ref{App:values}, and we obtain $\L(\M_1,\Lambda_0)=2$.
\end{proof}

Combining the previous theorem with Theorem~\ref{thm:invariance_complement}, we deduce the following corollary.

\begin{corollary}\label{cor:oZP}
	If $i\neq j \in\{1,2,3,4\}$, then there does not exist an ordered and oriented homeomorphism from $M(\M_i)$ to $M(\M_j)$.
\end{corollary}

\subsection{Eleven lines Zariski pair with non-isomorphic fundamental groups}\label{sec:ZP}\mbox{}

In this Section, we use a similar argument as in~\cite{ACCM,Gue:ZP} to produce a Zariski pair. Consider the arrangement $\mathfrak{M}_i=\M_i\cup L^i_{11}$, where $L^i_{11}$ is the line defined by $(\alpha^2+2\alpha+1)x + (\alpha^3-\alpha-2)y + z=0$ (if the considered arrangement $\M_i$ is obvious, the exponent $i$ will sometime be omitted). This line passes through the points $P_{8,9}$ and $P_{5,10}$ of $\M_i$. Notice that the line $L^i_{11}$ intersects the lines $L_k\in\M_i$, for $k\in\{1,2,3,4,5,7\}$, outside $\Sing(\M_i)$. In other words, the addition of the line $L^i_{11}$ to the arrangement $\M_i$ creates only two dense singular points (both of multiplicity $3$). So the combinatorics $\mathfrak{C}$ shared by the $\mathfrak{M}_i$'s is:
\begin{align*}
	\mathfrak{C} = \big\{\, 
	& \{1, 2, 6\}, \{1, 3, 9\}, \{1, 4, 5\}, \{1, 7\}, \{1, 8, 10\}, \{1, 11\}, \{2, 3, 7\}, \{2, 4, 10\}, \\
	& \{2, 5, 9\}, \{2, 8\}, \{2, 11\}, \{3, 4, 8\}, \{3, 5\}, \{3, 6, 10\}, \{3, 11\}, \{4, 6\}, \{4, 7, 9\}, \\
	& \{4, 11\}, \{5, 6, 7, 8\}, \{5, 10, 11\}, \{6, 9\}, \{6, 11\}, \{7, 10\}, \{7, 11\}, \{8, 9, 11\}, \{9, 10\}\, \big\}.
\end{align*}

\begin{lemma}\label{lem:trivial_automorphism}
	The automorphism group of $\mathfrak{C}$ is trivial.
\end{lemma}

\begin{proof}
	The line $L_{11}$ is the only one of $\mathfrak{C}$ which contains only two dense singular points. So any automorphism of $\mathfrak{C}$ fixes $L_{11}$, thus the automorphism group of $\mathfrak{C}$ is a subgroup of $\Aut(C)$. The lines $L_5$ and $L_8$ are the only ones that contain 1 quadruple point and 3 triple points. So, any automorphism of $\mathfrak{C}$ fixes $\{L_5,L_8\}$. The only permutations of $\Aut(C)$ verifying this condition are $\Id$ and $\sigma^3$. Since $\Aut(\mathfrak{C})$ is a group, the only possibility is that the automorphism group is trivial. 
\end{proof}

\begin{theorem}\label{thm:ZP}
	Let $i\neq j \in\{1,2,3,4\}$, if $i+j\not\equiv 0 \mod 5$ then there does not exist a homeomorphism from $M(\mathfrak{M}_i)$ to $M(\mathfrak{M}_j)$.
\end{theorem}

\begin{proof}
	We have a natural injection\footnote{This can be viewed as a consequence of Theorem~\ref{thm:multiplicativity} presented below.} $\mu:\TLG(C,\ZZ/5\ZZ)\hookrightarrow\TLG(\mathfrak{C},\ZZ/5\ZZ)$ which verifies $\L(\mathfrak{M}_k,\mu(\Lambda_0))=\L(\M_k,\Lambda_0)$. By Lemma~\ref{lem:trivial_automorphism}, the automorphism group of $\mathfrak{C}$ is trivial, then we conclude using Remark~\ref{rmk:trivial_automorphism}.
\end{proof}

\begin{theorem}\label{thm:grp_fonda}
	Let $i\neq j \in\{1,2,3,4\}$, if $i+j\not\equiv 0 \mod 5$ then
	\begin{equation*}
		\pi_1(M(\mathfrak{M}_i)) \not\simeq \pi_1(M(\mathfrak{M}_j)).
	\end{equation*}
\end{theorem}

The proof of the previous theorem is exactly the same as the one for Theorem~4.4 in~\cite{ACGM}. Therefore, here we only give a sketch of the proof containing the major steps and arguments.

\begin{proof}[Sketch of the proof]\mbox{}\\
	$\bullet$ \textsc{Step 1.} The combinatorics of $\mathfrak{M}_i$ is homologically rigid (obtained using~\cite[Corollary~3]{Marco}).\\
	\indent\quad 1.1. the combinatorics has enough triangles.\\
	\indent\quad 1.2. the combinatorics contains only pencils of point type.\\
	\indent\quad 1.3. the combinatorics is stronlgy connected.\\
	$\bullet$ \textsc{Step 2.} Apply the AI-isomorphism test to obtain an obstruction on the existence of an isomorphism respecting the homological structure of the fundamental groups (see~\cite[Section~4.2]{ACGM} for more details about this test).
\end{proof}

\begin{rmk}
	To our knowledge, only one other example of arithmetic Zariski pair with non-isomorphic fundamental groups is currently known (see~\cite{ACGM}), even in the more general case of plane curves. Notice that the example provided in the present paper has one fewer line than the example of~\cite{ACGM}.
\end{rmk}

\subsection{Twelve lines Zariski triple with non-isomorphic fundamental groups}\mbox{}

The proofs of this section are similar to those of Section~\ref{sec:oZP} and~\ref{sec:ZP}. So, to avoid unecessary details, we give here only the statements of the results.\\

Let $\xi$ be a fixed root of the 7th cyclotomic polynomial $Z^6+Z^5+Z^4+Z^3+Z^2+Z+1$. We consider the arrangements $\N_i$, for $i\in\{1,\dots,6\}$, defined by the following equations, where $\alpha=\xi^i$.
\begin{equation*}
	\begin{array}{l p{1cm} l}
		L_1: z = 0 && 
		L_2: \alpha x - y +z = 0 \\
		L_3: (\alpha^2+\alpha)x - \alpha y + (\alpha+1) z = 0 &&
		L_4: x = 0 \\
		L_5: x + \alpha y -(\alpha+1) z = 0 &&
		L_6: (\alpha+1)x - y = 0 \\
		L_7: x - z = 0 &&
		L_8: \alpha x + \alpha^2 y -(\alpha^3+\alpha^2+\alpha) z = 0 \\
		L_9: y = 0 &&
		L_{10}: y - z = 0 \\
		L_{11}: (\alpha^2+1) y -(\alpha^2+\alpha+1) z = 0 && 
	\end{array}
\end{equation*}
The combinatorics shared by the $\N_i$'s is:
\begin{align*}
	D = \big\{
	& \{1, 2\}, \{1, 3, 6\}, \{1, 4, 7\}, \{1, 5, 8\}, \{1, 9, 10, 11\}, \{2, 3, 9\}, \{2, 4, 10\}, \{2, 5, 11\}, \{2, 6, 7, 8\}, \\
	& \{3, 4, 5\}, \{3, 7\}, \{3, 8, 11\}, \{3, 10\}, \{4, 6, 9\}, \{4, 8\}, \{4, 11\}, \{5, 6\}, \{5, 7, 10\}, \{5, 9\}, \{6, 10\}, \\
	& \{6, 11\}, \{7, 9\}, \{7, 11\}, \{8, 9\}, \{8, 10\}
	\big\}.
\end{align*}

\begin{rmk}
	The combinatorics $C$ of Section~\ref{sec:oZP} is not a sub-combinatorics of $D$.
\end{rmk}

\begin{lemma}
	The automorphism group of $D$ is the subgroup of the symetric group $\Sigma_{11}$ generated by the permutations:
	\begin{equation*}
		\sigma_1 = (3,4,5)(6,7,8)(9,10,11) \quad\text{and}\quad \sigma_2 = (1,2)(6,9)(7,10)(8,11).
	\end{equation*}
	Their actions on the lines of $D$ is given by $\sigma_i(L_j)=L_{\sigma_i(j)}$.
\end{lemma}

\begin{theorem}
	The tensor linking group $\TLG(D,\ZZ/7\ZZ)$ is isomorphic to $\ZZ/7\ZZ$. Furthermore, if $\Lambda_0$ is a non-trivial element of $\TLG(D,\ZZ/7\ZZ)$, then 
	\begin{equation*}
		\L(\N_i,\Lambda_0) \neq 0.
	\end{equation*}
\end{theorem}

The previous theorem is obtained using the braided wiring diagrams given in Figures~\ref{fig:WD_1},~\ref{fig:WD_2} and~\ref{fig:WD_3}, by removing the strand associated to $L_{12}$. 

\begin{corollary}
	If $i\neq j \in\{1,\dots,6\}$, then there does not exist an ordered and oriented homeomorphism from $M(\N_i)$ to $M(\N_j)$.
\end{corollary}

As in Section~\ref{sec:ZP}, to remove the "ordered and oriented" condition in the previous corollary, we add to $\N_i$ a 12th line $L_{12}^i$ passing through $L_1 \cap L_9 \cap L_{10} \cap L_{11}$ and $L_3 \cap L_7$. We denote by $\mathfrak{N}_i$ the arrangement $\N_i \cup\{L_{12}^i\}$. Due to this twelveth line, the group of automorphism of the combinatorics $\mathfrak{D}$ shared by the $\mathfrak{N}_i$'s is trivial. Then, using Remark~\ref{rmk:trivial_automorphism}, we deduce the following theorem. We give in Figures~\ref{fig:WD_1},~\ref{fig:WD_2} and~\ref{fig:WD_3} a non-generic braided wiring diagram of $\mathfrak{N}_1$, $\mathfrak{N}_2$ and $\mathfrak{N}_3$ respectively, where $L_1$ is considered as the line at infinity and $\xi \simeq 0.62 + 0.78i$.

\begin{figure}[h!]
	\begin{tikzpicture}
	\begin{scope}[xscale=0.25,yscale=-0.4]

		\node at (-1,1) {$L_{9}$};
		\node at (-1,2) {$L_{12}$};
		\node at (-1,3) {$L_{11}$};
		\node at (-1,4) {$L_{10}$};
		\node at (-1,5) {$L_{2}$};
		\node at (-1,6) {$L_{4}$};
		\node at (-1,7) {$L_{7}$};
		\node at (-1,8) {$L_{5}$};
		\node at (-1,9) {$L_{6}$};
		\node at (-1,10) {$L_{8}$};
		\node at (-1,11) {$L_{3}$};
		\nocrossing{11}{0}
		\actualcrossing{11}{1}{4}{3}
		\actualcrossing{11}{3}{6}{3}
		\actualcrossing{11}{5}{8}{2}
		\actualcrossing{11}{6}{9}{2}
		\actualcrossing{11}{7}{10}{2}
		\undercrossing{11}{8}{5}
		\undercrossing{11}{9}{8}
		\actualcrossing{11}{10}{6}{4}
		\undercrossing{11}{13}{8}
		\undercrossing{11}{14}{7}
		\undercrossing{11}{15}{8}
		\undercrossing{11}{16}{6}
		\overcrossing{11}{17}{5}
		\actualcrossing{11}{18}{3}{2}
		\actualcrossing{11}{19}{4}{3}
		\actualcrossing{11}{21}{6}{2}
		\actualcrossing{11}{22}{7}{2}
		\actualcrossing{11}{23}{8}{3}
		\overcrossing{11}{25}{5}
		\actualcrossing{11}{26}{2}{2}
		\actualcrossing{11}{27}{3}{2}
		\actualcrossing{11}{28}{4}{2}
		\actualcrossing{11}{29}{5}{2}
		\actualcrossing{11}{30}{6}{3}
		\actualcrossing{11}{32}{8}{2}
		\overcrossing{11}{33}{7}
		\undercrossing{11}{34}{2}
		\actualcrossing{11}{35}{1}{2}
		\actualcrossing{11}{36}{2}{3}
		\actualcrossing{11}{38}{4}{3}
		\actualcrossing{11}{40}{6}{2}
		\actualcrossing{11}{41}{7}{2}
		\overcrossing{11}{42}{5}
		\actualcrossing{11}{43}{1}{2}
		\undercrossing{11}{44}{5}
		\overcrossing{11}{45}{2}
		\actualcrossing{11}{46}{2}{3}
		\overcrossing{11}{48}{2}
		\overcrossing{11}{49}{5}
		\overcrossing{11}{50}{6}
		\actualcrossing{11}{51}{4}{2}
		\nocrossing{11}{52}

	\end{scope}
\end{tikzpicture}
	\vspace{-0.5cm}
	\caption{Braided wiring diagram of $\mathfrak{N}_1$.\label{fig:WD_1}}
\end{figure}
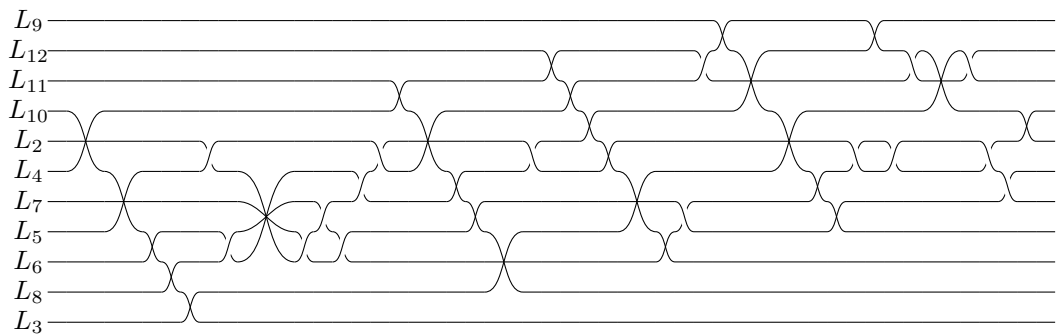

\begin{figure}[h!]
	\begin{tikzpicture}
	\begin{scope}[xscale=0.25,yscale=-0.4]

		\node at (-1,1) {$L_{9}$};
		\node at (-1,2) {$L_{10}$};
		\node at (-1,3) {$L_{11}$};
		\node at (-1,4) {$L_{12}$};
		\node at (-1,5) {$L_{6}$};
		\node at (-1,6) {$L_{2}$};
		\node at (-1,7) {$L_{7}$};
		\node at (-1,8) {$L_{3}$};
		\node at (-1,9) {$L_{4}$};
		\node at (-1,10) {$L_{5}$};
		\node at (-1,11) {$L_{8}$};
		\nocrossing{11}{0}
		\actualcrossing{11}{1}{8}{3}
		\undercrossing{11}{3}{9}
		\overcrossing{11}{4}{5}
		\overcrossing{11}{5}{7}
		\overcrossing{11}{6}{10}
		\overcrossing{11}{7}{8}
		\actualcrossing{11}{8}{4}{2}
		\actualcrossing{11}{9}{5}{2}
		\actualcrossing{11}{10}{6}{2}
		\actualcrossing{11}{11}{7}{3}
		\actualcrossing{11}{13}{9}{2}
		\actualcrossing{11}{14}{10}{2}
		\undercrossing{11}{15}{9}
		\actualcrossing{11}{16}{9}{2}
		\overcrossing{11}{17}{9}
		\overcrossing{11}{18}{8}
		\undercrossing{11}{19}{4}
		\actualcrossing{11}{20}{3}{2}
		\actualcrossing{11}{21}{4}{3}
		\actualcrossing{11}{23}{6}{2}
		\actualcrossing{11}{24}{7}{2}
		\actualcrossing{11}{25}{8}{3}
		\undercrossing{11}{27}{5}
		\actualcrossing{11}{28}{3}{2}
		\undercrossing{11}{29}{4}
		\undercrossing{11}{30}{3}
		\actualcrossing{11}{31}{2}{3}
		\actualcrossing{11}{33}{4}{2}
		\actualcrossing{11}{34}{5}{3}
		\actualcrossing{11}{36}{7}{2}
		\actualcrossing{11}{37}{8}{2}
		\overcrossing{11}{38}{2}
		\undercrossing{11}{39}{6}
		\undercrossing{11}{40}{7}
		\actualcrossing{11}{41}{1}{2}
		\actualcrossing{11}{42}{2}{2}
		\actualcrossing{11}{43}{3}{3}
		\actualcrossing{11}{45}{5}{2}
		\actualcrossing{11}{46}{6}{3}
		\undercrossing{11}{48}{3}
		\undercrossing{11}{49}{1}
		\undercrossing{11}{50}{4}
		\undercrossing{11}{51}{5}
		\overcrossing{11}{52}{4}
		\overcrossing{11}{53}{2}
		\overcrossing{11}{54}{3}
		\actualcrossing{11}{55}{2}{4}
		\nocrossing{11}{58}

	\end{scope}
\end{tikzpicture}
	\vspace{-0.5cm}
	\caption{Braided wiring diagram of $\mathfrak{N}_2$.\label{fig:WD_2}}
\end{figure}
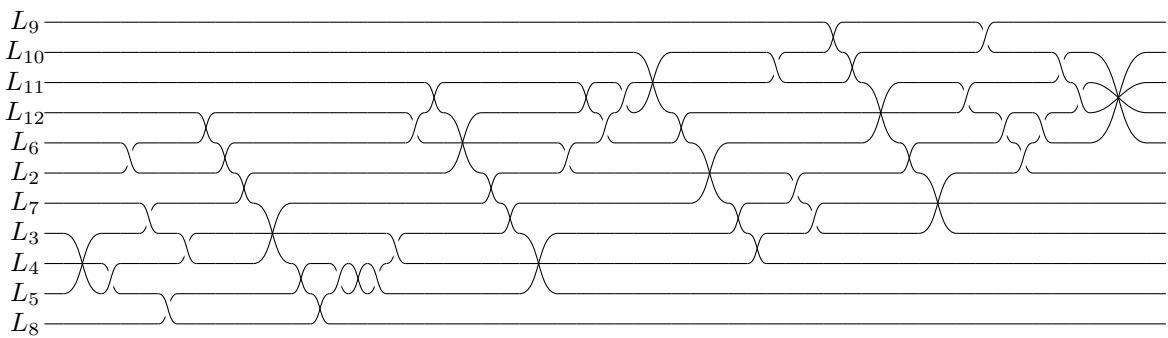

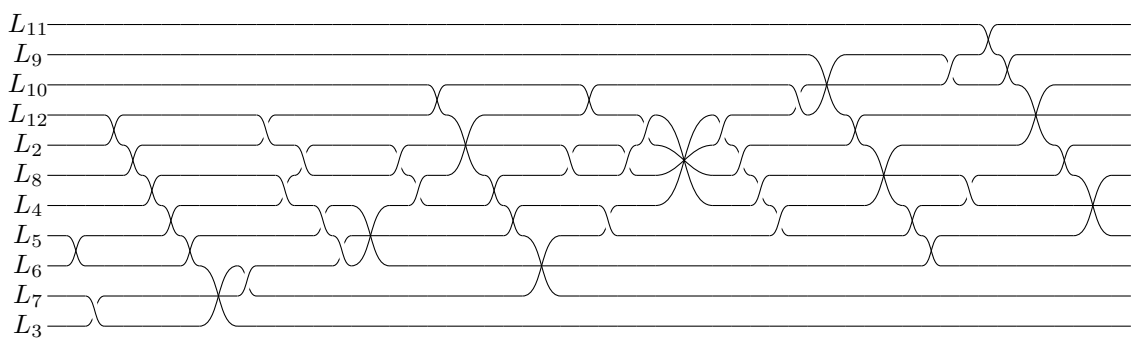
\begin{figure}[h!]
	\begin{tikzpicture}
	\begin{scope}[xscale=0.25,yscale=-0.4]

		\node at (-1,1) {$L_{11}$};
		\node at (-1,2) {$L_{9}$};
		\node at (-1,3) {$L_{10}$};
		\node at (-1,4) {$L_{12}$};
		\node at (-1,5) {$L_{2}$};
		\node at (-1,6) {$L_{8}$};
		\node at (-1,7) {$L_{4}$};
		\node at (-1,8) {$L_{5}$};
		\node at (-1,9) {$L_{6}$};
		\node at (-1,10) {$L_{7}$};
		\node at (-1,11) {$L_{3}$};
		\nocrossing{11}{0}
		\actualcrossing{11}{1}{8}{2}
		\overcrossing{11}{2}{10}
		\actualcrossing{11}{3}{4}{2}
		\actualcrossing{11}{4}{5}{2}
		\actualcrossing{11}{5}{6}{2}
		\actualcrossing{11}{6}{7}{2}
		\actualcrossing{11}{7}{8}{2}
		\actualcrossing{11}{8}{9}{3}
		\undercrossing{11}{10}{9}
		\overcrossing{11}{11}{4}
		\overcrossing{11}{12}{6}
		\overcrossing{11}{13}{5}
		\overcrossing{11}{14}{7}
		\overcrossing{11}{15}{8}
		\actualcrossing{11}{16}{7}{3}
		\undercrossing{11}{18}{5}
		\undercrossing{11}{19}{6}
		\actualcrossing{11}{20}{3}{2}
		\actualcrossing{11}{21}{4}{3}
		\actualcrossing{11}{23}{6}{2}
		\actualcrossing{11}{24}{7}{2}
		\actualcrossing{11}{25}{8}{3}
		\overcrossing{11}{27}{5}
		\actualcrossing{11}{28}{3}{2}
		\overcrossing{11}{29}{7}
		\undercrossing{11}{30}{5}
		\undercrossing{11}{31}{4}
		\actualcrossing{11}{32}{4}{4}
		\undercrossing{11}{35}{4}
		\undercrossing{11}{36}{5}
		\undercrossing{11}{37}{6}
		\undercrossing{11}{38}{7}
		\overcrossing{11}{39}{3}
		\actualcrossing{11}{40}{2}{3}
		\actualcrossing{11}{42}{4}{2}
		\actualcrossing{11}{43}{5}{3}
		\actualcrossing{11}{45}{7}{2}
		\actualcrossing{11}{46}{8}{2}
		\undercrossing{11}{47}{2}
		\overcrossing{11}{48}{6}
		\actualcrossing{11}{49}{1}{2}
		\actualcrossing{11}{50}{2}{2}
		\actualcrossing{11}{51}{3}{3}
		\actualcrossing{11}{53}{5}{2}
		\actualcrossing{11}{54}{6}{3}
		\nocrossing{11}{56}

	\end{scope}
\end{tikzpicture}
	\vspace{-0.5cm}
	\caption{Braided wiring diagram of $\mathfrak{N}_3$.\label{fig:WD_3}}
\end{figure}

\begin{theorem}
	For $i \neq j \in \{1,\dots,6\}$, if $i+j \not\equiv 0 \mod 7$ then there does not exist a homeomorphism from $M(\mathfrak{N}_i)$ to $M(\mathfrak{N}_j)$.
\end{theorem} 

In addition, we can use the homological rigidity of $\mathfrak{D}$ and the AI-test as done for Theorem~\ref{thm:grp_fonda}, to prove that the fundamental groups of the $\mathfrak{N}_i$'s are not isomorphic.

\begin{theorem}
	For $i \neq j \in \{1,\dots,6\}$, if $i+j \not\equiv 0 \mod 7$ then 
	\begin{equation*}
		\pi_1\big(M(\mathfrak{N}_i)\big) \not\simeq \pi_1\big(M(\mathfrak{N}_j)\big).
	\end{equation*}
\end{theorem}

In other words, the arrangements $\mathfrak{N}_1$, $\mathfrak{N}_2$ and $\mathfrak{N}_3$ form an arithmetic Zariski triple whose complements have non-isomorphic fundamental groups.

\section{Multiplicativity theorem of the loop linking number}
	
In this section, we present a generalization of the multiplicativity theorem obtained in~\cite{Gue:multiplicativity}. Let $\A_1=\{L_1^1,\dots,L_n^1\}$ and $\A_2=\{L_1^2,\dots,L_m^2\}$ be two ordered\footnote{We consider the orders given by the indices, i.e. $L^k_i < L^k_j \Leftrightarrow i<j$.} line arrangements such that for a fixed $r\in\{0,\dots,\min(n,m)\}$,
\begin{enumerate}
	\item for all $i\leq r$, $L_{i}^1=L_i^2$,
	\item for all $r<i$ and all $r<j$, $L_i^1\notin \A_2$ and $L_j^2\notin \A_1$.
\end{enumerate}

\begin{rmk}\mbox{}
	\begin{enumerate}
		\item If $r=0$ then Condition~(1) is trivial and Condition~(2) can be reformulated as follows: $\A_1$ and $\A_2$ have no line in common.
		\item There is no restriction on the intersection of $\A_1$ and $\A_2$.
	\end{enumerate}
\end{rmk}

\begin{definition}
	The \emph{ordered union} of $\A_1$ and $\A_2$, denoted by $\A_1\ocup\A_2$, is the ordered arrangement $\{L_1^1,\dots,L_n^1,L_{r+1}^2,\dots,L_m^2\}$, with the unique order coinciding with the one of $\A_1$ (resp. $\A_2$) on $L_1^1,\dots,L_n^1$ (resp. $L_{r+1}^2,\dots,L_m^2$) and such that $L_i^1 < L_j^2$ for $i\in\{1,\dots,n\}$ and $j\in\{r+1,\dots,m\}$.
\end{definition}

\Warning As ordered arrangements, $\A_1\ocup\A_2$ and $\A_2\ocup\A_1$ are different; albeit as arrangements (without order) they are equal.

\subsection{Multiplicativity theorem}\mbox{}

Let $f_i^*:\HH^1(M(\A_i);G)\rightarrow\HH^1(M(\A_1\ocup\A_2);G)$ be the map induced by the inclusion $f_i$ of $M(\A_1\ocup\A_2)$ in $M(\A_i)$; similarly, let ${g_i}_*:\HH_1(\Gamma(\A_i))\rightarrow\HH_1(\Gamma(\A_1\ocup\A_2))$ be the map induced by the inclusion of $\Gamma(\A_i)$ in $\Gamma(\A_1\ocup\A_2)$.
So the map $f_i^*\otimes{g_i}_*$ is a morphism defined by
\begin{equation*}
	f_i^*\otimes{g_i}_* : \left\{
	\begin{array}{rccc}
		\HH^1(M(\A_i);G)\otimes_\ZZ\HH_1(\Gamma(\A_i),\ZZ) & \longrightarrow & \HH^1(M(\A_1\ocup\A_2);G)\otimes_\ZZ\HH_1(\Gamma(\A_1\ocup\A_2),\ZZ) \\
		\lambda_\PtoL \otimes \PtoL & \longmapsto & f_i^*(\lambda_\PtoL) \otimes {g_i}_*(\PtoL)
	\end{array}
	\right..
\end{equation*}
If $\Lambda_i$ is an element in $\HH^1(M(\A_i);G)\otimes\HH_1(\Gamma(\A_i))$, then we define
\begin{equation*}
	\Lambda_1 \oplus \Lambda_2 = f_1^*\otimes{g_1}_* (\Lambda_1) + f_2^*\otimes{g_2}_* (\Lambda_2). 
\end{equation*} 

\begin{proposition}
	If $\Lambda_1\in\TLG(\A_1,G)$ and $\Lambda_2\in\TLG(\A_2,G)$, then the tensor $\Lambda_1 \oplus \Lambda_2$ is an element of $\TLG(\A_1\ocup\A_2,G)$.
\end{proposition}

\begin{proof}
	To be an element of $\TLG(\A_1\ocup\A_2,G)$, the tensor $\Lambda_1\oplus\Lambda_2$ needs to verify the two following conditions:
	\begin{enumerate}
		\item For all $\PtoL\in\Gamma\big(\A_1\ocup\A_2\big)$, and $L'\in\A_1\ocup\A_2$ containing $P$, we have $\lambda_\PtoL(L')=0_G$,
		\item For all $\PtoL\in\Gamma\big(\A_1\ocup\A_2\big)$, and $P'\in\Sing(\A_1\ocup\A_2)$ contained in $L$, we have 
		$$\sum_{L'\ni P'}\lambda_\PtoL(L')=0_G.$$
	\end{enumerate}
	By linearity, these conditions are sums of the same conditions for $\Lambda_1$ and $\Lambda_2$, so it follows that $\Lambda_1\oplus\Lambda_2$ verifies the two previous conditions.
\end{proof}

\begin{lemma}\label{lem:technique}
	For $k\in\{1,2\}$ we have the following equality:
	\begin{equation*}
		\pi_{1,2} \circ \Psi_{1,2} (f_k^*\otimes{g_k}_* (\Lambda_k) ) = \pi_{k} \circ \Psi_{k} (\Lambda_k). 
	\end{equation*}
\end{lemma}

\begin{proof}
	We denote by $h_{k}$ (resp. $h_{1,2}$) the map $i\circ j$ described in Section~\ref{sec:construction} sending $\Gamma(\A_k)$ in $M(\A_k)$ (resp. $\Gamma\big(\A_1\ocup\A_2\big)$ in $M(\A_1\ocup\A_2)$).
	We can choose $h_{1,2}$ such that $f_k \circ h_{1,2} \circ g_k (\Gamma(\A_k))$ is contained in $B(\A_k)$ and is a coherent embedding. We fix $h_k=f_k \circ h_{1,2} \circ g_k$, so we have the following commutative diagram
	\begin{center}
		\begin{tikzcd}
			\Gamma(\A_k) \arrow[r, "h_k"] \arrow[d, "g_k", hook]
			& M(\A_k) \\
			\Gamma\big(\A_1\ocup\A_2\big) \arrow[r,"h_{1,2}"] 
			& M(\A_1\ocup\A_2) \arrow[u,"f_k", hook]
		\end{tikzcd}
	\end{center}
	The lemma is then a consequence of Theorem~\ref{thm:computation}. Indeed, in the first hand, the braid $B^k_{\PtoL}\in\A_k$ can be obtained from the braid $B^{1,2}_{\PtoL}$ in $\A_1\ocup\A_2$ by conserving only the strands associated to the lines of $\A_k$. On the other hand, let $\Lambda_k = \sum_{P\in\Sing(\A_k)} \sum_{L\ni P} \lambda^{k}_\PtoL \otimes \PtoL$ and	$f_k^*\otimes{g_k}_* (\Lambda_k) = \sum_{P\in\Sing(\A_1\ocup\A_2)} \sum_{L\ni P} \lambda^{1,2}_\PtoL \otimes \PtoL$. If $\PtoL$ is not an edge of $\Gamma(\A_k)$ then, by definition, $\lambda^{1,2}_{\PtoL}$ is trivial; and if $\PtoL$ is an edge of $\Gamma(\A_k)$, then $\lambda^{1,2}_{\PtoL}$ acts trivially on the meridians of the lines of $(\A_1\ocup\A_2)\setminus\A_k$ and $\lambda^{1,2}_{\PtoL}(m_L)=\lambda^{k}_{\PtoL}(m_L)$ for any $L\in\A_k$. It follows that
	\begin{equation*}
		\lambda^{1,2}_{\PtoL}\big(\ulk_L(B^{1,2}_{\PtoL})\big) =
		\left\{
			\begin{array}{ll}
				\lambda^{k}_{\PtoL}\big(\ulk_L(B^k_{\PtoL})\big),\quad & \text{if } \PtoL\in\Gamma(\A_k),\\
				0, & \text{otherwise.} 
			\end{array}
		\right.
		\qedhere
	\end{equation*}
\end{proof}

\begin{theorem}\label{thm:multiplicativity}
	Let $\Lambda_1\in\TLG(\A_1,G)$ and $\Lambda_2\in\TLG(\A_2,G)$, we have
	\begin{equation*}
		\L(\A_1\ocup\A_2,\Lambda_1\oplus\Lambda_2) = \L(\A_1,\Lambda_1) + \L(\A_2,\Lambda_2).
	\end{equation*}
\end{theorem}

\begin{proof}
	To clarify the notation, we write with an index $1,2$ (resp. $1$ and $2$) the maps associated to the arrangement $\A_1\ocup\A_2$ (resp. $\A_1$ and $\A_2$). Using Lemma~\ref{lem:technique}, the theorem comes from the following computation.
	\begin{align*}
		\L(\A_1\ocup\A_2,\Lambda_1\oplus\Lambda_2) 
		& = \pi_{1,2} \circ \Psi_{1,2} (f_1^*\otimes{g_1}_* (\Lambda_1) + f_2^*\otimes{g_2}_* (\Lambda_2)),\\
		& = \pi_{1,2} \circ \Psi_{1,2} (f_1^*\otimes{g_1}_* (\Lambda_1) ) + \pi_{1,2} \circ \Psi_{1,2} (f_2^*\otimes{g_2}_* (\Lambda_2)),\\
		& = \pi_{1} \circ \Psi_{1} (\Lambda_1) + \pi_{2} \circ \Psi_{2} (\Lambda_2),\\
		& = \L(\A_1,\Lambda_1) + \L(\A_2,\Lambda_2).
		\qedhere
	\end{align*}
\end{proof}

\subsection{Application to Rybnikov's arrangements}\mbox{}

	Let $\ML^+$ and $\ML^-$ be two complex conjugated realizations of the MacLane matroid (see~\cite{McL}). 
	
	\begin{lemma}[\cite{Cad}]\label{lem:technique1}
		The tensor linking group $\TLG(\ML^\pm,\ZZ/3\ZZ)$ is isomorphic to $\ZZ/3\ZZ$, and for any non-trivial $\Lambda\in\TLG(\ML^\pm,\ZZ/3\ZZ)$, the loop linking number $\L(\ML^+,\Lambda)$ is not $0$.
	\end{lemma}
	
	In the following of this subsection, we fix $\Lambda_0$ the generator of $\TLG(\ML^\pm,\ZZ/3\ZZ)$ verifying
	\begin{equation*}
		\L(\ML^+,\Lambda_0) = 1, \quad \text{and} \quad \L(\ML^-,\Lambda_0) = 2.
	\end{equation*}
	
	Rybnikov's arrangements $\R^+$ and $\R^-$ are constructed by gluing together two copies of the positive MacLane arrangements for $\R^+$ and a copy of the positive with a copy of negative MacLane arrangements for $\R^-$ (see~\cite{Ryb,ACCM:Rybnikov}). These gluings can be described as follows. Let $\ell_0,\ell_1,\ell_2$ be three concurrent lines in the MacLane matroid, and let $\ell^\pm_0,\ell^\pm_1,\ell^\pm_2$ be their realizations in $\ML^\pm$. Let $\psi^+$ (resp. $\psi^-$) be a generic linear map such that for $i\in\{1,2,3\}$, $\psi^+(\ell^+_i)=\ell^+_i$ (resp. $\psi^-(\ell^-_i)=\ell^+_i$). Rybnikov's arrangements are then defined as $\R^+=\ML^+\ocup\psi^+(\ML^+)$ and $\R^-=\ML^+\ocup\psi^-(\ML^-)$. 
	
	\begin{lemma}\label{lem:technique3}
		Let $\R^+$ and $\R^-$ be the two Rybnikov arrangements, we have
		\begin{equation*}
			\L(\R^+,\Lambda_0 \oplus \Lambda_0) = 2, \quad \text{and} \quad \L(\R^-,\Lambda_0 \oplus \Lambda_0) = 0.
		\end{equation*}
	\end{lemma}
	
	\begin{proof}
		By construction, Rybnikov's arrangements are ordered unions of two MacLane arrangements. By Theorem~\ref{thm:multiplicativity}, we have 
		\begin{align*}
			\L(\R^\pm,\Lambda_0 \oplus \Lambda_0) 
			& = \L\big(\ML^+ \ocup \psi^\pm(\ML^\pm) , \Lambda_0 \oplus \Lambda_0\big),\\
			& = \L\big(\ML^+,\Lambda_0\big) + \L\big(\psi^\pm(\ML^\pm),\Lambda_0\big),\\
			& = \L(\ML^+,\Lambda_0) + \L(\ML^\pm,\Lambda_0).
		\end{align*}
		Then we have $\L(\R^+,\Lambda_0 \oplus \Lambda_0) = 2$ and $\L(\R^-,\Lambda_0 \oplus \Lambda_0) = 0$.
	\end{proof}
	
	\begin{theorem}\label{thm:Rybnikov_linking}
		Let $\R^+$ and $\R^-$ be the two Rybnikov arrangements, we have
		\begin{equation*}
			\Lb(\R^+,\Lambda_0 \oplus \Lambda_0) = \{1,2\}, \quad \text{and} \quad \Lb(\R^-,\Lambda_0 \oplus \Lambda_0) = \{0\}.
		\end{equation*}
	\end{theorem}
	
	\begin{proof}
		As noted in~\cite{ACCM:Rybnikov} (as a consequence of Proposition~1.11), the automorphism group of $\Aut(\R^\pm)$ is isomorphic to $\Sigma_3\times\ZZ/2\ZZ$. The first part permutes the lines $\{\ell^\pm_0,\ell^\pm_1,\ell^\pm_2\}$ and the second part fixes or exchanges the MacLane arrangements. Since we are considering the same tensor element $\Lambda_0$ on the two copies of the MacLane arrangment then for any $\sigma\in\Sigma_3$ we have:
		\begin{equation*}
			\L\big( (\sigma\times 0) \cdot\R^\pm,\Lambda_0\oplus\Lambda_0 \big) = \L\big( (\sigma\times 1) \cdot\R^\pm,\Lambda_0\oplus\Lambda_0 \big).
		\end{equation*}
		By Theorem~\ref{thm:multiplicativity}, we have $\L\big( (\sigma\times 0) \cdot \R^\pm,\Lambda_0\oplus\Lambda_0) = \L\big(\sigma\cdot\ML^+,\Lambda_0\big) + \L\big(\sigma\cdot\psi^\pm(\ML^\pm),\Lambda_0\big)$.	Since $\psi^\pm$ are in $\PGL_3(\CC)$, then they are homeomorphisms of $\CC\PP^2$. Thus $\L(\sigma\cdot\A^\pm,\Lambda_0)=\L(\sigma\cdot\psi^\pm(\A^\pm),\Lambda_0)$. This implies that
		\begin{equation*}
			\L\big( (\sigma\times 0) \cdot\R^\pm,\Lambda_0\oplus\Lambda_0 \big) = \L\big(\sigma\cdot\ML^+,\Lambda_0\big) + \L\big(\sigma\cdot\ML^\pm,\Lambda_0\big).
		\end{equation*}
		As a consequence of Lemma~\ref{lem:technique1}, we have $\L\big( (\sigma\times 0) \cdot \R^+,\Lambda_0\oplus\Lambda_0)\subset \{1,2\}$. By Lemma~\ref{lem:technique3} and the definition of the full loop linking number, the inclusion is an equality. Since $\overline{\sigma\cdot\ML^+}=\sigma\cdot\ML^-$, then by Proposition~\ref{prop:Galois_invariance}, $\L\big( (\sigma\times 0) \cdot \R^-,\Lambda_0\oplus\Lambda_0)=0$.
	\end{proof}
	
	
	\begin{corollary}
		The complements of Rybnikov's arrangements are not homeomorphic.
	\end{corollary}

\begin{rmk}\mbox{}
	\begin{enumerate}
		\item This result is in adequation with Rybnikov's results~\cite{Ryb}, who proved that the fundamental groups of $M(\R^+)$ and $M(\R^-)$ are not isomorphic. Nevertheless, and to our knowledge, this is the first proof, without computer assistance, that these complements are not homeomorphic. 
		\item	Even if such a multiplicativity theorem exists for the $\I$-invariant (see~\cite{Gue:multiplicativity}), it does not allows to detect a difference in Rybnikov's arrangements because of too strong combinatorial conditions imposed by the $\I$-invariant.
	\end{enumerate}
\end{rmk}

\subsection{Rybnikov-like Zariski pairs}\mbox{}

We can generalize Rybnikov's result by the following theorem.

\begin{theorem}\label{thm:Rybnikov-like}
	Let $\A$ be an arrangement. If $\Lambda_0\in\TLG(\A,G)$ and $\psi^+,\psi^-\in\PGL_3(\CC)$ are such that:
	\begin{enumerate}[label=(\roman*)]
		\item $\Aut(\A\ocup\psi^+(\A))=H^+ \times \ZZ/2\ZZ$ (resp. $\Aut(\A\ocup\psi^-(\overline{A}))= H^- \times \ZZ/2\ZZ$), with $H^+$ (resp. $H^-$) a subgroup of $\Aut(\A)=\Aut(\overline{\A})$ and $\ZZ/2\ZZ$ fixes or exchanges $\A$ with $\psi^+(\A)$ (resp. $\psi^-(\overline{\A})$),
		\item for any $h \in H^+ \cup H^-$
			\begin{equation*}
				2.\L(h\cdot\A,\Lambda_0) \in G\setminus\{0_G\},
			\end{equation*}
	\end{enumerate}
	then the complements of $\mathfrak{A}^+=\A\ocup\psi^+(\A)$ and $\mathfrak{A}^-=\A\ocup\psi^-(\overline{\A})$ are not homeomorphic.
\end{theorem}

\begin{proof}
	Basically, the proof is the same as for Theorem~\ref{thm:Rybnikov_linking}. By Condition~(i), for any $h^\pm\in H^\pm$
	\begin{equation*}
		\L\big((h^\pm\times 0)\cdot\mathfrak{A}^\pm,\Lambda_0\oplus\Lambda_0\big) 
		= \L\big((h^\pm\times 1)\cdot\mathfrak{A}^\pm,\Lambda_0\oplus\Lambda_0\big).
	\end{equation*}
	By Theorem~\ref{thm:multiplicativity} and Proposition~\ref{prop:Galois_invariance}, we get 
	\begin{equation*}
		\L\big((h^\pm\times 0)\cdot\mathfrak{A}^\pm,\Lambda_0\oplus\Lambda_0\big) = \L\big(h^\pm\cdot\A,\Lambda_0\big) \pm \L\big(h^\pm\cdot \A,\Lambda_0 \big). 
	\end{equation*}
	So $\L\big((h^-\times 0)\cdot\mathfrak{A}^-,\Lambda_0\oplus\Lambda_0\big)=0_G$; and by Condition~(ii), we have $\L\big((h^+\times 0)\cdot\mathfrak{A}^+,\Lambda_0\oplus\Lambda_0\big) \subset G\setminus \{0_G\}$. As for Rybnikov's arrangements, this induces that $\overline{\L}(\mathfrak{A}^-,\Lambda_0\oplus\Lambda_0)=\{0_G\}$ and $\overline{\L}(\mathfrak{A}^+,\Lambda_0\oplus\Lambda_0)\subset G\setminus\{0\}$. We conclude using Theorem~\ref{thm:full_loop_linking}.	
\end{proof}

\begin{corollary}
	Let $\A$, $\Lambda_0$, $\psi^+$ and $\psi^-$ be as in Theorem~\ref{thm:Rybnikov-like}. If, in addition, $\mathfrak{A}^+=\A\ocup\psi^+(\A)$ and $\mathfrak{A}^-=\A\ocup\psi^-(\overline{\A})$ have isomorphic intersection lattices, then they form a Zariski pair with non-homeomorphic complements.
\end{corollary}

\subsection{Homotopy and topology of the complement}\mbox{}

In~\cite{Gue:homotopy}, the author constructs examples of homotopy-equivalent Zariski pairs as follows. Let $\A_1=\{L_1^1,\dots,L_n^1\}$ and $\A_2=\{L_1^2,\dots,L_n^2\}$ be two real complexified line arrangements, with isomorphic ordered combinatorics (with orders given by the indices). We assume that their intersection is generic (this is always possible up to an action of $\PGL_3(\CC)$). For $i\in\{1,2\}$, let $T_i=\{D_1^i,D_2^i\}$ be an arrangement such that $D_j^i\cap\Sing(\A_1\cup\A_2)=\emptyset$ and $D_1^i\cap D_2^i\cap L_1^i\neq\emptyset$. We define the ordered arrangements $\mathfrak{A}_1=T_1\ocup\A_1\ocup\A_2$ and $\mathfrak{A}_2=T_2\ocup\A_2\ocup\A_1$. By~\cite[Theorem~1.4]{Gue:homotopy}, if $\A_1$ and $\A_2$ form a Zariski pair, then $\mathfrak{A}_1$ and $\mathfrak{A}_2$ form a homotopy-equivalent Zariski pair.

\begin{theorem}\label{thm:homotopy_vs_homeomorphism}
	It exists lattice-isomorphic arrangements with non-homeomorphic and homotopy-equivalent complements.
\end{theorem}

\begin{proof}
	Let $\A_1=\{L_1^1,\dots,L_{13}^1\}$ and $\A_2=\{L_1^2,\dots,L_{13}^2\}$ be the Zariski pair with 13 lines and only double and triple points given in~\cite{GBVS}, and $\mathfrak{A}_1$ and $\mathfrak{A}_2$ constructed from $\A_1$ and $\A_2$ as in the first paragraph of this section. As previously noted, $\mathfrak{A}_1$ and $\mathfrak{A}_2$ are lattice-isomorphic arrangements with homotopy-equivalent complements. In order to complete the proof, we need to prove that these complements are not homeomorphic.
	
	In~\cite{GBVS}, the arrangements $\A_1$ and $\A_2$ are distinguished using the $\I$-invariant. We consider $\xi\in\HH^1(M(\A_i);\ZZ/2\ZZ)$ and $\gamma\in\HH_1(\Gamma(\A_i);\ZZ)$ such that $(\A_i,\xi,\gamma)$ is an inner-cyclic triple, or equivalently such that the tensor $\Lambda_0=\xi\otimes\gamma\in\TLG(\A_i,\ZZ/2\ZZ)$. From~\cite[proof of Theorem~2.6]{GBVS}, we have
	\begin{equation*}
		\L(\A_1,\Lambda_0)\equiv 1\mod 2 \quad\text{and}\quad \L(\A_2,\Lambda_0) \equiv 0\mod 2.
	\end{equation*}
	
	The combinatorics shared by the $\A_i$'s is such that 3 lines contain all the triple points. We assume that these lines are $L_1^i$, $L_2^i$ and $L_3^i$. This implies that the set $\{L_1^i,L_2^i,L_3^i\}$ is fixed by any automorphism of the combinatorics. By the description of the automorphism group of the combinatorics of the $\A_i$'s given in~\cite[Remark~2.5]{GBVS}, we know that only one line of $\{L_1^i,L_2^i,L_3^i\}$ is fixed by all the automorphisms of the combinatorics, the two others being fixed solely by the identity. We assume that this fixed line is $L_3^i$. This implies that $\Aut(T_i\ocup\A_i)\simeq\ZZ/2\ZZ$ where the action fixes or exchanges $D_1^i$ and $D_2^i$. Note that this action always fixes $\A_i$ line by line. Since $T_i\ocup\A_i$ and $\A_{i+1}$ (where the indices are considered modulo 2) are both connected arrangements (see~\cite{Fan} for the definition), it follows that $\Aut(\mathfrak{A}_i)\simeq\ZZ/2\ZZ\times\Aut(\A_{i+1})$. The action of $(k,\sigma)\in\Aut(\mathfrak{A}_i)$ fixes or exchanges $D_1^i$ and $D_2^i$ by $k$, fixes $\A_i$, and acts on $\A_{i+1}$ by~$\sigma$. This description, together with Theorem~\ref{thm:multiplicativity}, implies that for all $(k,\sigma)\in\Aut(\mathfrak{A}_i)$,
	\begin{align*}
		\L\big((k,\sigma)\cdot\mathfrak{A}_i,\ 0\otimes\Lambda_0\otimes 0\big) 
		& = \L\left( (k\cdot T_i) \ocup \A_i \ocup (\sigma\cdot\A_{i+1}),\ 0\otimes\Lambda_0\otimes 0\right), \\
		& = \L(k\cdot T_i, 0) + \L(\A_i,\Lambda_0) + \L(\sigma\cdot\A_{i+1},0), \\
		& = \L(\A_i,\Lambda_0).
	\end{align*}
	By Corollary~\ref{cor:full_loop_linking}, we obtain that $\mathfrak{A}_1$ and $\mathfrak{A}_2$ have non-homeomorphic complements.
\end{proof}

As a consequence of the previous proof, we obtain the following corollary.

\begin{corollary}\label{cor:not_homotopy_determined}
	The loop linking number and the full loop linking number are not determined by the homotopy type of the complement. 
\end{corollary}

\begin{rmk}
	The proof of Theorem~\ref{thm:homotopy_vs_homeomorphism} also implies that the $\I$-invariant is not determined by the homotopy type of the complement.
\end{rmk}

\section{Discussions}
	
During our research on the linking properties of line arrangements, in particular those which leads to the present paper, we noticed that in the wide range of examples computed, the loop linking number seems to be trivial as soon as $G$ has no torsion element. So, we conjecture the following.\\[-10pt]

\noindent\textbf{Conjecture.} \emph{Let $G$ be an Abelian group without torsion and $\A$ a complex line arrangement. For any $\Lambda\in\TLG(\A,G)$, the loop linking number of $\A$ associated to the tensor $\Lambda$ is trivial.}\\[-10pt]

\noindent Note that in particular, this conjecture implies that if a tensor element $\Lambda_0$ in $\TLG(\A,\ZZ/n\ZZ)$ arises as a reduction of an element of $\TLG(\A,\ZZ)$ then $\L(\A,\Lambda_0) = 0$.\\

In addition to the previous conjecture and to the regard of the high number of Zariski pairs which share the two properties: being Galois conjugated in the $5$th cyclotomic field (or in an isomorphic field), and having a tensor linking group with coefficients in $\ZZ/5\ZZ$, it seems natural to discuss the following question.\\[-10pt]

\noindent\textbf{Question.} \emph{For a prime number $p\geq 5$, let $\A_1\dots,\A_{p-1}$ be arrangements conjugated in the $p$-th cyclotomic fields and such that $\TLG(\A_i,\ZZ/p\ZZ)$ is not trivial. If $\Lambda_0\in\TLG(\A_i,\ZZ/p\ZZ)$ does not arised from the reduction of an integral tensor element, then do we necessarly have $\L(\A_i,\Lambda_0)\neq 0$?}\\[-10pt]

The fact that the definition field contains naturally all the $p$-roots of unity is mendatory by Proposition~\ref{prop:Galois_invariance}. To solve this question and/or the conjecture could help to understand the behavior of the branched Galois covers of $\CC\PP^2$.
	
\section*{Acknowledgements}
	During this work the author have been supported, first by the postdoctoral grant \#2017/15369-0 of the Funda\c c\~ao de Amparo \`a Pesquisa do Estado de S\~ao Paulo (FAPESP), and second by the Polish Academy of Sciences.
	
	The author would like to thank \textsc{J.~Viu-Sos} for all their rewarding discussions and all his remarks/comments which greatly contributed to improve the quality of this article. He is also grateful to \textsc{E.~Hollen} for her efficient and careful proofreading of this manuscript.

	
\appendix
		\newpage
\section{Values for the arrangements $\M_i$}\label{App:values}

\subsection{Generator of $\TLG(C,\ZZ/5\ZZ)$}\mbox{}

The generator $\Lambda_0$ of $\TLG(C,\ZZ/5\ZZ)$ is given by the following characters $\lambda_{\PtoL}$.

\begin{equation*}
\begin{aligned}[t]
(0, 0, 0, 0, 0, 0, 0, 0, 0, 0)\ = & \ \lambda_{ P_{1,2,6} \rightarrow  L_{1} } \\ 
(0, 1, 0, 0, 0, 4, 0, 4, 0, 1)\ = & \ \lambda_{ P_{1,3,9} \rightarrow  L_{1} } \\ 
(0, 4, 2, 0, 0, 1, 0, 1, 3, 4)\ = & \ \lambda_{ P_{1,4,5} \rightarrow  L_{1} } \\ 
(0, 2, 3, 3, 2, 3, 0, 0, 2, 0)\ = & \ \lambda_{ P_{1,7} \rightarrow  L_{1} } \\ 
(0, 3, 0, 2, 3, 2, 0, 0, 0, 0)\ = & \ \lambda_{ P_{1,8,10} \rightarrow  L_{1} } \\ 
 \\  
(0, 0, 0, 0, 0, 0, 0, 0, 0, 0)\ = & \ \lambda_{ P_{1,2,6} \rightarrow  L_{2} } \\ 
(1, 0, 0, 2, 3, 4, 0, 0, 2, 3)\ = & \ \lambda_{ P_{2,3,7} \rightarrow  L_{2} } \\ 
(0, 0, 3, 0, 2, 0, 2, 0, 3, 0)\ = & \ \lambda_{ P_{2,4,10} \rightarrow  L_{2} } \\ 
(1, 0, 0, 0, 0, 4, 0, 0, 0, 0)\ = & \ \lambda_{ P_{2,5,9} \rightarrow  L_{2} } \\ 
(3, 0, 2, 3, 0, 2, 3, 0, 0, 2)\ = & \ \lambda_{ P_{2,8} \rightarrow  L_{2} } \\ 
 \\  
(0, 4, 0, 4, 0, 1, 1, 1, 0, 4)\ = & \ \lambda_{ P_{1,3,9} \rightarrow  L_{3} } \\ 
(4, 0, 0, 0, 0, 3, 0, 0, 1, 2)\ = & \ \lambda_{ P_{2,3,7} \rightarrow  L_{3} } \\ 
(4, 4, 0, 0, 0, 1, 1, 0, 1, 4)\ = & \ \lambda_{ P_{3,4,8} \rightarrow  L_{3} } \\ 
(2, 2, 0, 3, 0, 0, 3, 2, 3, 0)\ = & \ \lambda_{ P_{3,5} \rightarrow  L_{3} } \\ 
(0, 0, 0, 3, 0, 0, 0, 2, 0, 0)\ = & \ \lambda_{ P_{3,6,10} \rightarrow  L_{3} } \\ 
 \\  
(0, 4, 3, 0, 0, 0, 1, 2, 4, 1)\ = & \ \lambda_{ P_{1,4,5} \rightarrow  L_{4} } \\ 
(2, 0, 2, 0, 3, 0, 3, 3, 2, 0)\ = & \ \lambda_{ P_{2,4,10} \rightarrow  L_{4} } \\ 
(3, 1, 0, 0, 2, 0, 1, 0, 4, 4)\ = & \ \lambda_{ P_{3,4,8} \rightarrow  L_{4} } \\ 
(0, 0, 0, 0, 0, 0, 0, 0, 0, 0)\ = & \ \lambda_{ P_{4,6} \rightarrow  L_{4} } \\ 
(0, 0, 0, 0, 0, 0, 0, 0, 0, 0)\ = & \ \lambda_{ P_{4,7,9} \rightarrow  L_{4} } \\ 
 \\  
(0, 2, 0, 0, 0, 4, 4, 2, 3, 0)\ = & \ \lambda_{ P_{1,4,5} \rightarrow  L_{5} } \\ 
(4, 0, 0, 1, 0, 1, 4, 0, 0, 0)\ = & \ \lambda_{ P_{2,5,9} \rightarrow  L_{5} } \\ 
(3, 3, 0, 2, 0, 0, 2, 3, 2, 0)\ = & \ \lambda_{ P_{3,5} \rightarrow  L_{5} } \\ 
(3, 0, 0, 2, 0, 0, 0, 0, 0, 0)\ = & \ \lambda_{ P_{5,6,7,8} \rightarrow  L_{5} } \\ 
(0, 0, 0, 0, 0, 0, 0, 0, 0, 0)\ = & \ \lambda_{ P_{5,10} \rightarrow  L_{5} } \\ 
\end{aligned}
\hspace{1.5cm}
\begin{aligned}[t] 
(0, 0, 0, 0, 0, 0, 0, 0, 0, 0)\ = & \ \lambda_{ P_{1,2,6} \rightarrow  L_{6} } \\ 
(3, 2, 0, 0, 0, 0, 0, 0, 0, 0)\ = & \ \lambda_{ P_{3,6,10} \rightarrow  L_{6} } \\ 
(0, 0, 0, 0, 0, 0, 0, 0, 0, 0)\ = & \ \lambda_{ P_{4,6} \rightarrow  L_{6} } \\ 
(2, 3, 0, 0, 0, 0, 0, 0, 0, 0)\ = & \ \lambda_{ P_{5,6,7,8} \rightarrow  L_{6} } \\ 
(0, 0, 0, 0, 0, 0, 0, 0, 0, 0)\ = & \ \lambda_{ P_{6,9} \rightarrow  L_{6} } \\ 
 \\  
(0, 3, 2, 2, 3, 2, 0, 0, 3, 0)\ = & \ \lambda_{ P_{1,7} \rightarrow  L_{7} } \\ 
(0, 0, 0, 3, 2, 3, 0, 0, 2, 0)\ = & \ \lambda_{ P_{2,3,7} \rightarrow  L_{7} } \\ 
(0, 0, 0, 0, 0, 0, 0, 0, 0, 0)\ = & \ \lambda_{ P_{4,7,9} \rightarrow  L_{7} } \\ 
(0, 2, 3, 0, 0, 0, 0, 0, 0, 0)\ = & \ \lambda_{ P_{5,6,7,8} \rightarrow  L_{7} } \\ 
(0, 0, 0, 0, 0, 0, 0, 0, 0, 0)\ = & \ \lambda_{ P_{7,10} \rightarrow  L_{7} } \\ 
 \\  
(0, 0, 0, 0, 2, 3, 0, 0, 0, 0)\ = & \ \lambda_{ P_{1,8,10} \rightarrow  L_{8} } \\ 
(2, 0, 3, 2, 0, 3, 2, 0, 0, 3)\ = & \ \lambda_{ P_{2,8} \rightarrow  L_{8} } \\ 
(3, 0, 0, 0, 3, 4, 3, 0, 0, 2)\ = & \ \lambda_{ P_{3,4,8} \rightarrow  L_{8} } \\ 
(0, 0, 2, 3, 0, 0, 0, 0, 0, 0)\ = & \ \lambda_{ P_{5,6,7,8} \rightarrow  L_{8} } \\ 
(0, 0, 0, 0, 0, 0, 0, 0, 0, 0)\ = & \ \lambda_{ P_{8,9} \rightarrow  L_{8} } \\ 
 \\  
(0, 0, 0, 1, 0, 0, 4, 0, 0, 0)\ = & \ \lambda_{ P_{1,3,9} \rightarrow  L_{9} } \\ 
(0, 0, 0, 4, 0, 0, 1, 0, 0, 0)\ = & \ \lambda_{ P_{2,5,9} \rightarrow  L_{9} } \\ 
(0, 0, 0, 0, 0, 0, 0, 0, 0, 0)\ = & \ \lambda_{ P_{4,7,9} \rightarrow  L_{9} } \\ 
(0, 0, 0, 0, 0, 0, 0, 0, 0, 0)\ = & \ \lambda_{ P_{6,9} \rightarrow  L_{9} } \\ 
(0, 0, 0, 0, 0, 0, 0, 0, 0, 0)\ = & \ \lambda_{ P_{8,9} \rightarrow  L_{9} } \\ 
(0, 0, 0, 0, 0, 0, 0, 0, 0, 0)\ = & \ \lambda_{ P_{9,10} \rightarrow  L_{9} } \\ 
 \\  
(0, 2, 0, 3, 0, 0, 0, 0, 0, 0)\ = & \ \lambda_{ P_{1,8,10} \rightarrow  L_{10} } \\ 
(3, 0, 0, 0, 0, 0, 0, 2, 0, 0)\ = & \ \lambda_{ P_{2,4,10} \rightarrow  L_{10} } \\ 
(2, 3, 0, 2, 0, 0, 0, 3, 0, 0)\ = & \ \lambda_{ P_{3,6,10} \rightarrow  L_{10} } \\ 
(0, 0, 0, 0, 0, 0, 0, 0, 0, 0)\ = & \ \lambda_{ P_{5,10} \rightarrow  L_{10} } \\ 
(0, 0, 0, 0, 0, 0, 0, 0, 0, 0)\ = & \ \lambda_{ P_{7,10} \rightarrow  L_{10} } \\ 
(0, 0, 0, 0, 0, 0, 0, 0, 0, 0)\ = & \ \lambda_{ P_{9,10} \rightarrow  L_{10} } \\ 
\end{aligned}
\end{equation*}

\newpage
\subsection{The upper linking numbers of $\M_1$}\mbox{}

The values of the upper linking numbers of the braid $B_{\PtoL}$ are given below. The values denoted with a $=^*$ symbol have been simplified using the property $\sum_{L\in\A} m_L = 0$.

\begin{equation*}
\begin{aligned}[t]
& \ulk( B_{ P_{1,2,6} \rightarrow  L_{1} } ) =  m_{8} \\
& \ulk( B_{ P_{1,3,9} \rightarrow  L_{1} } ) =  m_{8} \\
& \ulk( B_{ P_{1,4,5} \rightarrow  L_{1} } ) =  m_{2} + m_{6} + m_{7} + m_{8} \\
& \ulk( B_{ P_{1,7} \rightarrow  L_{1} } ) =  m_{2} + m_{6} + m_{8} \\
& \ulk( B_{ P_{1,8,10} \rightarrow  L_{1} } ) =  0  \\
 \\  
& \ulk( B_{ P_{1,2,6} \rightarrow  L_{2} } ) =  m_{8} \\
& \ulk( B_{ P_{2,3,7} \rightarrow  L_{2} } ) =  m_{1} + m_{6} + m_{8} \\
& \ulk( B_{ P_{2,4,10} \rightarrow  L_{2} } ) =  m_{1} + m_{5} + m_{6} + m_{7} + m_{8} \\
& \ulk( B_{ P_{2,5,9} \rightarrow  L_{2} } ) =  m_{1} + m_{3} + m_{6} + m_{7} + m_{8} \\
& \ulk( B_{ P_{2,8} \rightarrow  L_{2} } ) =  0  \\
 \\  
& \ulk( B_{ P_{1,3,9} \rightarrow  L_{3} } ) =  m_{4} + m_{8} \\
& \ulk( B_{ P_{2,3,7} \rightarrow  L_{3} } ) =^* - m_2 - m_3 - m_5 - m_7 \\
& \ulk( B_{ P_{3,4,8} \rightarrow  L_{3} } ) =  0  \\
& \ulk( B_{ P_{3,5} \rightarrow  L_{3} } ) =^* - m_3 - m_5 \\
& \ulk( B_{ P_{3,6,10} \rightarrow  L_{3} } ) =  m_{1} + m_{2} + m_{4} + m_{8} + m_{9} \\
 \\  
& \ulk( B_{ P_{1,4,5} \rightarrow  L_{4} } ) =  m_{6} + m_{7} + m_{8} + m_{9} \\
& \ulk( B_{ P_{2,4,10} \rightarrow  L_{4} } ) =^* - m_2 - m_3 - m_4 - m_{10} \\
& \ulk( B_{ P_{3,4,8} \rightarrow  L_{4} } ) =  0  \\
& \ulk( B_{ P_{4,6} \rightarrow  L_{4} } ) =  m_{8} + m_{9} \\
& \ulk( B_{ P_{4,7,9} \rightarrow  L_{4} } ) =  m_{6} + m_{8} \\
 \\  
& \ulk( B_{ P_{1,4,5} \rightarrow  L_{5} } ) =  m_{6} + m_{7} + m_{8} \\
& \ulk( B_{ P_{2,5,9} \rightarrow  L_{5} } ) =  m_{6} + m_{7} + m_{8} \\
& \ulk( B_{ P_{3,5} \rightarrow  L_{5} } ) =  m_{6} + m_{7} + m_{8} \\
& \ulk( B_{ P_{5,6,7,8} \rightarrow  L_{5} } ) =  0  \\
& \ulk( B_{ P_{5,10} \rightarrow  L_{5} } ) =  m_{6} + m_{7} + m_{8} \\
\end{aligned}
\hspace{0.5cm}
\begin{aligned}[t]
& \ulk( B_{ P_{1,2,6} \rightarrow  L_{6} } ) =  m_{8} \\
& \ulk( B_{ P_{3,6,10} \rightarrow  L_{6} } ) =  m_{8} \\
& \ulk( B_{ P_{4,6} \rightarrow  L_{6} } ) =  m_{8} \\
& \ulk( B_{ P_{5,6,7,8} \rightarrow  L_{6} } ) =  0  \\
& \ulk( B_{ P_{6,9} \rightarrow  L_{6} } ) =  m_{8} \\
 \\  
& \ulk( B_{ P_{1,7} \rightarrow  L_{7} } ) =  m_{6} + m_{8} \\
& \ulk( B_{ P_{2,3,7} \rightarrow  L_{7} } ) =  m_{6} + m_{8} \\
& \ulk( B_{ P_{4,7,9} \rightarrow  L_{7} } ) =  m_{6} + m_{8} \\
& \ulk( B_{ P_{5,6,7,8} \rightarrow  L_{7} } ) =  0  \\
& \ulk( B_{ P_{7,10} \rightarrow  L_{7} } ) =  m_{6} + m_{8} \\
 \\  
& \ulk( B_{ P_{1,8,10} \rightarrow  L_{8} } ) =  0  \\
& \ulk( B_{ P_{2,8} \rightarrow  L_{8} } ) =  0  \\
& \ulk( B_{ P_{3,4,8} \rightarrow  L_{8} } ) =  0  \\
& \ulk( B_{ P_{5,6,7,8} \rightarrow  L_{8} } ) =  0  \\
& \ulk( B_{ P_{8,9} \rightarrow  L_{8} } ) =  0  \\
 \\  
& \ulk( B_{ P_{1,3,9} \rightarrow  L_{9} } ) =  m_{8} \\
& \ulk( B_{ P_{2,5,9} \rightarrow  L_{9} } ) =^* - m_2 - m_5 - m_9 \\
& \ulk( B_{ P_{4,7,9} \rightarrow  L_{9} } ) =  m_{6} + m_{8} \\
& \ulk( B_{ P_{6,9} \rightarrow  L_{9} } ) =  m_{8} \\
& \ulk( B_{ P_{8,9} \rightarrow  L_{9} } ) =  0  \\
& \ulk( B_{ P_{9,10} \rightarrow  L_{9} } ) =  m_{1} + m_{4} + m_{6} + m_{7} + m_{8} \\
 \\  
& \ulk( B_{ P_{1,8,10} \rightarrow  L_{10} } ) =  0  \\
& \ulk( B_{ P_{2,4,10} \rightarrow  L_{10} } ) =  m_{1} + m_{5} + m_{6} + m_{7} + m_{8} \\
& \ulk( B_{ P_{3,6,10} \rightarrow  L_{10} } ) =  m_{1} + m_{2} + m_{8} \\
& \ulk( B_{ P_{5,10} \rightarrow  L_{10} } ) =  m_{1} + m_{6} + m_{7} + m_{8} \\
& \ulk( B_{ P_{7,10} \rightarrow  L_{10} } ) =  m_{1} + m_{6} + m_{8} \\
& \ulk( B_{ P_{9,10} \rightarrow  L_{10} } ) =  m_{1} + m_{6} + m_{7} + m_{8} \\
\end{aligned}
\end{equation*}

		\newpage
\section{Examples of arithmetic Zariski pairs with 11 lines}\label{App:examples}

\vfill

The following arrangements admit a non-trivial tensor linking group isomorphic to $\ZZ/5\ZZ$ and such that for any generator $\Lambda_0$ of $\TLG(C_k,\ZZ/5\ZZ)$, we have $\L(\B_k^i,\Lambda_0)\neq 0$. For all of them except the last one, the couple $(\B_k^i,\B_k^j)$ form a Zariski pair as soon as $i+j\not\equiv 0 \mod 5$. The last example is an ordered Zariski pair. To our knowledge, the following list completed with the examples obtained from the odered Zariski pair of Section~\ref{sec:oZP}, is the complet list of all the arithmetic Zariski pairs with 11 lines that can be detected by the loop linking number.

\vfill

\begin{spacing}{1.15}

\subsection{ }

\begin{itemize}[label=--,leftmargin=*]

\item Let $\zeta$ be a fixed root of the polynomial $Z^4 + Z^3 + Z^2 + Z + 1$. For $i\in\{1,2,3,4\}$, we define $\B_\ssnum^i$ by the following equations, where $\alpha=\zeta^i$.
\begin{equation*}
  \begin{array}{l p{0.5cm} l}
    L_{1} : z = 0 &&
    L_{2} : (-\alpha - 1)x + y = 0 \\
    L_{3} : x = 0 &&
    L_{4} : x - z = 0 \\
    L_{5} : (-\alpha^2 - \alpha - 1)x + z = 0 &&
    L_{6} : y = 0 \\
    L_{7} : x + (\alpha^3 + \alpha^2 + \alpha + 1)y = 0 &&
    L_{8} : y - z = 0 \\
    L_{9} : (-\alpha^3 - \alpha^2 - \alpha)x + y + (\alpha^3 + \alpha^2 + \alpha)z = 0 &&
    L_{10} : (-\alpha^3 - \alpha^2 - \alpha)x + y + (\alpha^2 + \alpha)z = 0 \\
    L_{11} : (\alpha)x - y + z = 0 &&
  \end{array}
\end{equation*}
\item The combinatorics shared by the $\B_\ssnum^i$'s is given by:
\begin{align*}
  C_\ssnum=\big\{
  & \{1, 2\}, \{1, 3, 4, 5\}, \{1, 6, 8\}, \{1, 7, 11\}, \{1, 9, 10\}, \{2, 3, 6, 7\}, \{2, 4, 11\}, \{2, 5, 10\}\\
  &  \{2, 8, 9\}, \{3, 8, 11\}, \{3, 9\}, \{3, 10\}, \{4, 6, 9\}, \{4, 7\}, \{4, 8\}, \{4, 10\}, \{5, 6\}, \{5, 7\}\\
  &  \{5, 8\}, \{5, 9, 11\}, \{6, 10\}, \{6, 11\}, \{7, 8, 10\}, \{7, 9\}, \{10, 11\}
  \big\}
\end{align*}
\item The group of automorphisms of $C_\ssnum$ is trivial.

\end{itemize}

\vfill

\subsection{ }

\begin{itemize}[label=--,leftmargin=*]

\item Let $\zeta$ be a fixed root of the polynomial $Z^4 + Z^3 + Z^2 + Z + 1$. For $i\in\{1,2,3,4\}$, we define $\B_\ssnum^i$ by the following equations, where $\alpha=\zeta^i$.
\begin{equation*}
  \begin{array}{l p{0.5cm} l}
    L_{1} : z = 0 &&
    L_{2} : x = 0 \\
    L_{3} : x - z = 0 &&
    L_{4} : x + (-\alpha^3 - \alpha^2)y = 0 \\
    L_{5} : y = 0 &&
    L_{6} : (\alpha)x + y = 0 \\
    L_{7} : y - z = 0 &&
    L_{8} : (-\alpha^3 - \alpha^2 - \alpha - 1)y + z = 0 \\
    L_{9} : (\alpha^2)x + (\alpha)y + z = 0 &&
    L_{10} : (-\alpha^2)x + y + (-\alpha^3 - \alpha^2 - \alpha - 1)z = 0 \\
    L_{11} : (\alpha^2)x - y + z = 0 &&
  \end{array}
\end{equation*}
\item The combinatorics shared by the $\B_\ssnum^i$'s is given by:
\begin{align*}
  C_\ssnum=\big\{
  & \{1, 2, 3\}, \{1, 4\}, \{1, 5, 7, 8\}, \{1, 6, 9\}, \{1, 10, 11\}, \{2, 4, 5, 6\}, \{2, 7, 11\}, \{2, 8\}, \{2, 9, 10\}\\
  &  \{3, 4, 9\}, \{3, 5\}, \{3, 6, 8\}, \{3, 7\}, \{3, 10\}, \{3, 11\}, \{4, 7, 10\}, \{4, 8, 11\}, \{5, 9, 11\}\\
  &  \{5, 10\}, \{6, 7\}, \{6, 10\}, \{6, 11\}, \{7, 9\}, \{8, 9\}, \{8, 10\}
  \big\}
\end{align*}
\item The group of automorphisms of $C_\ssnum$ is trivial.

\end{itemize}

\vfill
\newpage

\subsection{ }

\begin{itemize}[label=--,leftmargin=*]

\item Let $\zeta$ be a fixed root of the polynomial $Z^4 + Z^3 + Z^2 + Z + 1$. For $i\in\{1,2,3,4\}$, we define $\B_\ssnum^i$ by the following equations, where $\alpha=\zeta^i$.
\begin{equation*}
  \begin{array}{l p{0.5cm} l}
    L_{1} : z = 0 &&
    L_{2} : x = 0 \\
    L_{3} : x - z = 0 &&
    L_{4} : y = 0 \\
    L_{5} : y - z = 0 &&
    L_{6} : x + (-\alpha^3 - \alpha^2)y = 0 \\
    L_{7} : -x + (\alpha^3 + \alpha^2)y + z = 0 &&
    L_{8} : (\alpha^2)x - y + z = 0 \\
    L_{9} : (-\alpha^2)x + y + (-\alpha^3 - \alpha^2 - \alpha - 1)z = 0 &&
    L_{10} : (\alpha)x + y = 0 \\
    L_{11} : (\alpha^2)x + (\alpha)y + z = 0 &&
  \end{array}
\end{equation*}
\item The combinatorics shared by the $\B_\ssnum^i$'s is given by:
\begin{align*}
  C_\ssnum=\big\{
  & \{1, 2, 3\}, \{1, 4, 5\}, \{1, 6, 7\}, \{1, 8, 9\}, \{1, 10, 11\}, \{2, 4, 6, 10\}, \{2, 5, 8\}, \{2, 7\}, \{2, 9, 11\}\\
  &  \{3, 4, 7\}, \{3, 5\}, \{3, 6, 11\}, \{3, 8\}, \{3, 9\}, \{3, 10\}, \{4, 8, 11\}, \{4, 9\}, \{5, 6, 9\}, \{5, 7\}\\
  &  \{5, 10\}, \{5, 11\}, \{6, 8\}, \{7, 8, 10\}, \{7, 9\}, \{7, 11\}, \{9, 10\}
  \big\}
\end{align*}
\item The group of automorphisms of $C_\ssnum$ is trivial.

\end{itemize}

\subsection{ }

\begin{itemize}[label=--,leftmargin=*]

\item Let $\zeta$ be a fixed root of the polynomial $Z^4 + Z^3 + Z^2 + Z + 1$. For $i\in\{1,2,3,4\}$, we define $\B_\ssnum^i$ by the following equations, where $\alpha=\zeta^i$.
\begin{equation*}
  \begin{array}{l p{0.5cm} l}
    L_{1} : z = 0 &&
    L_{2} : x = 0 \\
    L_{3} : x - z = 0 &&
    L_{4} : y = 0 \\
    L_{5} : y - z = 0 &&
    L_{6} : (\alpha^3 + \alpha^2 + \alpha + 1)x + y = 0 \\
    L_{7} : -x + (\alpha)y + z = 0 &&
    L_{8} : (-\alpha^2 - \alpha - 1)x - y + z = 0 \\
    L_{9} : (\alpha^2 + \alpha + 1)x + y + (-\alpha^3 - \alpha^2 - \alpha - 1)z = 0 &&
    L_{10} : (\alpha)x + y = 0 \\
    L_{11} : x + (-\alpha^3 - \alpha^2 - \alpha - 1)y + (-\alpha^3 - 1)z = 0 &&
  \end{array}
\end{equation*}
\item The combinatorics shared by the $\B_\ssnum^i$'s is given by:
\begin{align*}
  C_\ssnum=\big\{
  & \{1, 2, 3\}, \{1, 4, 5\}, \{1, 6, 7\}, \{1, 8, 9\}, \{1, 10, 11\}, \{2, 4, 6, 10\}, \{2, 5, 8\}, \{2, 7, 9\}\\
  &  \{2, 11\}, \{3, 4, 7\}, \{3, 5\}, \{3, 6, 11\}, \{3, 8\}, \{3, 9\}, \{3, 10\}, \{4, 8, 11\}, \{4, 9\}, \{5, 6, 9\}\\
  &  \{5, 7\}, \{5, 10\}, \{5, 11\}, \{6, 8\}, \{7, 8, 10\}, \{7, 11\}, \{9, 10\}, \{9, 11\}
  \big\}
\end{align*}
\item The group of automorphisms of $C_\ssnum$ is trivial.

\end{itemize}

\subsection{ }

\begin{itemize}[label=--,leftmargin=*]

\item Let $\zeta$ be a fixed root of the polynomial $Z^4 + 3Z^3 + 4Z^2 + 2Z + 1$. For $i\in\{1,2,3,4\}$, we define $\B_\ssnum^i$ by the following equations, where $\alpha=\zeta^i$.
\begin{equation*}
  \begin{array}{l p{0.5cm} l}
    L_{1} : z = 0 &&
    L_{2} : x = 0 \\
    L_{3} : x - z = 0 &&
    L_{4} : y = 0 \\
    L_{5} : y - z = 0 &&
    L_{6} : (\alpha^3 + 3\alpha^2 + 3\alpha)x - y + z = 0 \\
    L_{7} : (\alpha^2 + 2\alpha + 1)x + (\alpha)y + z = 0 &&
    L_{8} : (\alpha + 1)x + y = 0 \\
    L_{9} : (\alpha + 1)x + y + (-\alpha^3 - 3\alpha^2 - 4\alpha - 2)z = 0 &&
    L_{10} : -x + (\alpha)y + z = 0 \\
    L_{11} : x + (-\alpha)y + (-\alpha^2 - \alpha - 1)z = 0 &&
  \end{array}
\end{equation*}
\item The combinatorics shared by the $\B_\ssnum^i$'s is given by:
\begin{align*}
  C_\ssnum=\big\{
  & \{1, 2, 3\}, \{1, 4, 5\}, \{1, 6, 7\}, \{1, 8, 9\}, \{1, 10, 11\}, \{2, 4, 8\}, \{2, 5, 6\}, \{2, 7, 9, 10\}\\
  &  \{2, 11\}, \{3, 4, 10\}, \{3, 5\}, \{3, 6, 9\}, \{3, 7\}, \{3, 8, 11\}, \{4, 6, 11\}, \{4, 7\}, \{4, 9\}, \{5, 7, 8\}\\
  &  \{5, 9\}, \{5, 10\}, \{5, 11\}, \{6, 8\}, \{6, 10\}, \{7, 11\}, \{8, 10\}, \{9, 11\}
  \big\}
\end{align*}
\item The group of automorphisms of $C_\ssnum$ is trivial.

\end{itemize}

\subsection{ }

\begin{itemize}[label=--,leftmargin=*]

\item Let $\zeta$ be a fixed root of the polynomial $Z^4 + 3Z^3 + 4Z^2 + 2Z + 1$. For $i\in\{1,2,3,4\}$, we define $\B_\ssnum^i$ by the following equations, where $\alpha=\zeta^i$.
\begin{equation*}
  \begin{array}{l p{0.5cm} l}
    L_{1} : z = 0 &&
    L_{2} : x = 0 \\
    L_{3} : x - z = 0 &&
    L_{4} : y = 0 \\
    L_{5} : y - z = 0 &&
    L_{6} : -x + (\alpha)y + z = 0 \\
    L_{7} : (-\alpha^3 - 3\alpha^2 - 3\alpha - 1)x + (-\alpha^2 - \alpha - 1)y + z = 0 &&
    L_{8} : (\alpha^3 + 2\alpha^2 + \alpha - 1)x + y = 0 \\
    L_{9} : (-\alpha^3 - 3\alpha^2 - 3\alpha - 1)x + (\alpha)y + z = 0 &&
    L_{10} : x + (\alpha^3 + 3\alpha^2 + 3\alpha + 1)y = 0 \\
    L_{11} : (\alpha)x + (-\alpha^2 - \alpha - 1)y + z = 0 &&
  \end{array}
\end{equation*}
\item The combinatorics shared by the $\B_\ssnum^i$'s is given by:
\begin{align*}
  C_\ssnum=\big\{
  & \{1, 2, 3\}, \{1, 4, 5\}, \{1, 6, 7\}, \{1, 8, 9\}, \{1, 10, 11\}, \{2, 4, 8, 10\}, \{2, 5\}, \{2, 6, 9\}, \{2, 7, 11\}\\
  &  \{3, 4, 6\}, \{3, 5\}, \{3, 7\}, \{3, 8, 11\}, \{3, 9\}, \{3, 10\}, \{4, 7, 9\}, \{4, 11\}, \{5, 6, 11\}, \{5, 7\}\\
  &  \{5, 8\}, \{5, 9, 10\}, \{6, 8\}, \{6, 10\}, \{7, 8\}, \{7, 10\}, \{9, 11\}
  \big\}
\end{align*}
\item The group of automorphisms of $C_\ssnum$ is trivial.

\end{itemize}

\subsection{ }

\begin{itemize}[label=--,leftmargin=*]

\item Let $\zeta$ be a fixed root of the polynomial $Z^4 + Z^3 + Z^2 + Z + 1$. For $i\in\{1,2,3,4\}$, we define $\B_\ssnum^i$ by the following equations, where $\alpha=\zeta^i$.
\begin{equation*}
  \begin{array}{l p{0.5cm} l}
    L_{1} : z = 0 &&
    L_{2} : x = 0 \\
    L_{3} : x - z = 0 &&
    L_{4} : y = 0 \\
    L_{5} : y - z = 0 &&
    L_{6} : (-\alpha^2)x + y + (\alpha^3 + \alpha^2)z = 0 \\
    L_{7} : -x + (\alpha^3)y + z = 0 &&
    L_{8} : (\alpha^3)x + y = 0 \\
    L_{9} : (\alpha)x + (\alpha^3)y + z = 0 &&
    L_{10} : x + (-\alpha^3 - 1)y = 0 \\
    L_{11} : (\alpha)x + (\alpha^3 + \alpha^2 + 1)y + z = 0 &&
  \end{array}
\end{equation*}
\item The combinatorics shared by the $\B_\ssnum^i$'s is given by:
\begin{align*}
  C_\ssnum=\big\{
  & \{1, 2, 3\}, \{1, 4, 5\}, \{1, 6, 7\}, \{1, 8, 9\}, \{1, 10, 11\}, \{2, 4, 8, 10\}, \{2, 5\}, \{2, 6, 11\}, \{2, 7, 9\}\\
  &  \{3, 4, 7\}, \{3, 5\}, \{3, 6, 8\}, \{3, 9\}, \{3, 10\}, \{3, 11\}, \{4, 6\}, \{4, 9, 11\}, \{5, 6, 9\}, \{5, 7, 10\}\\
  &  \{5, 8\}, \{5, 11\}, \{6, 10\}, \{7, 8\}, \{7, 11\}, \{8, 11\}, \{9, 10\}
  \big\}
\end{align*}
\item The group of automorphisms of $C_\ssnum$ is trivial.

\end{itemize}

\subsection{ }

\begin{itemize}[label=--,leftmargin=*]

\item Let $\zeta$ be a fixed root of the polynomial $Z^4 - Z^3 + Z^2 - Z + 1$. For $i\in\{1,2,3,4\}$, we define $\B_\ssnum^i$ by the following equations, where $\alpha=\zeta^i$.
\begin{equation*}
  \begin{array}{l p{0.5cm} l}
    L_{1} : z = 0 &&
    L_{2} : x = 0 \\
    L_{3} : x - z = 0 &&
    L_{4} : y = 0 \\
    L_{5} : y - z = 0 &&
    L_{6} : (-\alpha^2 - 1)x + (\alpha^2)y + z = 0 \\
    L_{7} : (\alpha^3 - 1)x + y + (-\alpha^3 + \alpha^2 - \alpha + 1)z = 0 &&
    L_{8} : x + (-\alpha)y = 0 \\
    L_{9} : -x + (\alpha)y + z = 0 &&
    L_{10} : x + (-\alpha^3 + \alpha^2 - \alpha)y = 0 \\
    L_{11} : (\alpha^3 - \alpha^2 - 1)x + (\alpha^2)y + z = 0 &&
  \end{array}
\end{equation*}
\item The combinatorics shared by the $\B_\ssnum^i$'s is given by:
\begin{align*}
  C_\ssnum=\big\{
  & \{1, 2, 3\}, \{1, 4, 5\}, \{1, 6, 7\}, \{1, 8, 9\}, \{1, 10, 11\}, \{2, 4, 8, 10\}, \{2, 5\}, \{2, 6, 11\}, \{2, 7, 9\}\\
  &  \{3, 4, 9\}, \{3, 5, 6\}, \{3, 7\}, \{3, 8\}, \{3, 10\}, \{3, 11\}, \{4, 6\}, \{4, 7, 11\}, \{5, 7\}, \{5, 8, 11\}\\
  &  \{5, 9\}, \{5, 10\}, \{6, 8\}, \{6, 9, 10\}, \{7, 8\}, \{7, 10\}, \{9, 11\}
  \big\}
\end{align*}
\item The group of automorphisms of $C_\ssnum$ is trivial.

\end{itemize}

\subsection{ }

\begin{itemize}[label=--,leftmargin=*]

\item Let $\zeta$ be a fixed root of the polynomial $Z^4 + 3Z^3 + 4Z^2 + 2Z + 1$. For $i\in\{1,2,3,4\}$, we define $\B_\ssnum^i$ by the following equations, where $\alpha=\zeta^i$.
\begin{equation*}
  \begin{array}{l p{0.5cm} l}
    L_{1} : z = 0 &&
    L_{2} : x = 0 \\
    L_{3} : x - z = 0 &&
    L_{4} : y = 0 \\
    L_{5} : y - z = 0 &&
    L_{6} : (\alpha^2+\alpha+1)x - (\alpha) y + (\alpha)z = 0 \\ 
    L_{7} : -x + (\alpha^2 + 2\alpha + 1)y + z = 0 &&
    L_{8} : (-\alpha^2 + 1)x + (\alpha)y = 0 \\ 
    L_{9} : (-\alpha^2 - 2\alpha - 2)x + (\alpha)y + z = 0 &&
    L_{10} : x + (\alpha^3 + 2\alpha^2 + \alpha)y = 0 \\
    L_{11} : (\alpha^2 + \alpha + 1)x + (\alpha^2)y + (\alpha)z = 0 && 
  \end{array}
\end{equation*}
\item The combinatorics shared by the $\B_\ssnum^i$'s is given by:
\begin{align*}
  C_\ssnum=\big\{
  & \{1, 2, 3\}, \{1, 4, 5\}, \{1, 6, 7\}, \{1, 8, 9\}, \{1, 10, 11\}, \{2, 4, 8, 10\}, \{2, 5, 6\}, \{2, 7\}, \{2, 9, 11\}\\
  &  \{3, 4, 7\}, \{3, 5\}, \{3, 6, 9\}, \{3, 8\}, \{3, 10\}, \{3, 11\}, \{4, 6, 11\}, \{4, 9\}, \{5, 7\}, \{5, 8\}\\
  &  \{5, 9, 10\}, \{5, 11\}, \{6, 8\}, \{6, 10\}, \{7, 8, 11\}, \{7, 9\}, \{7, 10\}
  \big\}
\end{align*}
\item The group of automorphisms of $C_\ssnum$ is trivial.

\end{itemize}

\subsection{ }

\begin{itemize}[label=--,leftmargin=*]

\item Let $\zeta$ be a fixed root of the polynomial $Z^4 + Z^3 + Z^2 + Z + 1$. For $i\in\{1,2,3,4\}$, we define $\B_\ssnum^i$ by the following equations, where $\alpha=\zeta^i$.
\begin{equation*}
  \begin{array}{l p{0.5cm} l}
    L_{1} : z = 0 &&
    L_{2} : x = 0 \\
    L_{3} : x - z = 0 &&
    L_{4} : y = 0 \\
    L_{5} : y - z = 0 &&

    L_{6} : (\alpha)x - y + z = 0 \\
    L_{7} : -x + (-\alpha^3 - \alpha^2 - \alpha - 1)y + z = 0 &&
    L_{8} : x + (\alpha^3 + \alpha^2 + \alpha)y = 0 \\
    L_{9} : x + (\alpha^3 + \alpha^2 + \alpha)y + (-\alpha^3 - \alpha^2)z = 0 &&
    L_{10} : x + (\alpha)y = 0 \\
    L_{11} : (\alpha)x + (\alpha^2)y + z = 0 &&
  \end{array}
\end{equation*}
\item The combinatorics shared by the $\B_\ssnum^i$'s is given by:
\begin{align*}
  C_\ssnum=\big\{
  & \{1, 2, 3\}, \{1, 4, 5\}, \{1, 6, 7\}, \{1, 8, 9\}, \{1, 10, 11\}, \{2, 4, 8, 10\}, \{2, 5, 6\}, \{2, 7\}, \{2, 9, 11\}\\
  &  \{3, 4, 7\}, \{3, 5\}, \{3, 6, 9\}, \{3, 8\}, \{3, 10\}, \{3, 11\}, \{4, 6, 11\}, \{4, 9\}, \{5, 7, 8\}, \{5, 9, 10\}\\
  &  \{5, 11\}, \{6, 8\}, \{6, 10\}, \{7, 9\}, \{7, 10\}, \{7, 11\}, \{8, 11\}
  \big\}
\end{align*}
\item The group of automorphisms of $C_\ssnum$ is trivial.

\end{itemize}

\subsection{ }

\begin{itemize}[label=--,leftmargin=*]

\item Let $\zeta$ be a fixed root of the polynomial $Z^4 + 3Z^3 + 4Z^2 + 2Z + 1$. For $i\in\{1,2,3,4\}$, we define $\B_\ssnum^i$ by the following equations, where $\alpha=\zeta^i$.
\begin{equation*}
  \begin{array}{l p{0.5cm} l}
    L_{1} : z = 0 &&
    L_{2} : x = 0 \\
    L_{3} : x - z = 0 &&
    L_{4} : y = 0 \\
    L_{5} : y - z = 0 &&
    L_{6} : (\alpha+1)x + (\alpha)y - (\alpha)z = 0\\ 
    L_{7} : (\alpha+1)x + (\alpha)y - (3\alpha+1)z = 0&& 
    L_{8} : x + (\alpha^3 + 2\alpha^2 + 2\alpha)y = 0 \\
    L_{9} : (\alpha)x + (-\alpha^3 - 2\alpha^2 - 2\alpha - 1)y + z = 0 &&
    L_{10} : (-\alpha^3 - 2\alpha^2 - 2\alpha)x + y = 0 \\
    L_{11} : (-\alpha^3 - 2\alpha^2 - 2\alpha)x + y + (-\alpha^2 - 2\alpha - 2)z = 0 &&
  \end{array}
\end{equation*}
\item The combinatorics shared by the $\B_\ssnum^i$'s is given by:
\begin{align*}
  C_\ssnum=\big\{
  & \{1, 2, 3\}, \{1, 4, 5\}, \{1, 6, 7\}, \{1, 8, 9\}, \{1, 10, 11\}, \{2, 4, 8, 10\}, \{2, 5, 6\}, \{2, 7, 9\}\\
  &  \{2, 11\}, \{3, 4, 7\}, \{3, 5\}, \{3, 6, 11\}, \{3, 8\}, \{3, 9\}, \{3, 10\}, \{4, 6\}, \{4, 9, 11\}, \{5, 7, 8\}\\
  &  \{5, 9\}, \{5, 10\}, \{5, 11\}, \{6, 8\}, \{6, 9, 10\}, \{7, 10\}, \{7, 11\}, \{8, 11\}
  \big\}
\end{align*}
\item The group of automorphisms of $C_\ssnum$ is trivial.

\end{itemize}

\subsection{ }

\begin{itemize}[label=--,leftmargin=*]

\item Let $\zeta$ be a fixed root of the polynomial $Z^4 - 2Z^3 + 4Z^2 - 3Z + 1$. For $i\in\{1,2,3,4\}$, we define $\B_\ssnum^i$ by the following equations, where $\alpha=\zeta^i$.
\begin{equation*}
  \begin{array}{l p{0.5cm} l}
    L_{1} : z = 0 &&
    L_{2} : x = 0 \\
    L_{3} : x - z = 0 &&
    L_{4} : y = 0 \\
    L_{5} : y - z = 0 &&
    L_{6} : -x - y + z = 0 \\
    L_{7} : x + y + (\alpha^3 - 2\alpha^2 + 4\alpha - 3)z = 0 &&
    L_{8} : (-\alpha)x - y + z = 0 \\
    L_{9} : (\alpha)x + y + (\alpha^3 - 2\alpha^2 + 3\alpha - 2)z = 0 &&
    L_{10} : x + (\alpha^3 - \alpha^2 + 2\alpha)y = 0 \\
    L_{11} : -x + (-\alpha^3 + \alpha^2 - 2\alpha)y + z = 0 &&
  \end{array}
\end{equation*}
\item The combinatorics shared by the $\B_\ssnum^i$'s is given by:
\begin{align*}
  C_\ssnum=\big\{
  & \{1, 2, 3\}, \{1, 4, 5\}, \{1, 6, 7\}, \{1, 8, 9\}, \{1, 10, 11\}, \{2, 4, 10\}, \{2, 5, 6, 8\}, \{2, 7\}, \{2, 9, 11\}\\
  &  \{3, 4, 6, 11\}, \{3, 5\}, \{3, 7, 9\}, \{3, 8\}, \{3, 10\}, \{4, 7, 8\}, \{4, 9\}, \{5, 7\}, \{5, 9, 10\}, \{5, 11\}\\
  &  \{6, 9\}, \{6, 10\}, \{7, 10\}, \{7, 11\}, \{8, 10\}, \{8, 11\}
  \big\}
\end{align*}
\item The group of automorphisms of $C_\ssnum$ is trivial.

\end{itemize}

\subsection{ }

\begin{itemize}[label=--,leftmargin=*]

\item Let $\zeta$ be a fixed root of the polynomial $Z^4 + 2Z^3 + 4Z^2 + 3Z + 1$. For $i\in\{1,2,3,4\}$, we define $\B_\ssnum^i$ by the following equations, where $\alpha=\zeta^i$.
\begin{equation*}
  \begin{array}{l p{0.5cm} l}
    L_{1} : z = 0 &&
    L_{2} : x = 0 \\
    L_{3} : x - z = 0 &&
    L_{4} : y = 0 \\
    L_{5} : y - z = 0 &&
    L_{6} : (-\alpha - 1)x + (\alpha^2 + \alpha)y + z = 0 \\
    L_{7} : -x + (\alpha)y + z = 0 &&
    L_{8} : (-\alpha - 1)x + (\alpha)y + z = 0 \\
    L_{9} : (-\alpha^3 - 2\alpha^2 - 4\alpha - 2)x - y + z = 0 &&
    L_{10} : x + (-\alpha^3 - 2\alpha^2 - 4\alpha - 2)y = 0 \\
    L_{11} : (-\alpha^3 - \alpha^2 - 3\alpha - 1)x - y + z = 0 &&
  \end{array}
\end{equation*}
\item The combinatorics shared by the $\B_\ssnum^i$'s is given by:
\begin{align*}
  C_\ssnum=\big\{
  & \{1, 2, 3\}, \{1, 4, 5\}, \{1, 6, 7\}, \{1, 8, 9\}, \{1, 10, 11\}, \{2, 4, 10\}, \{2, 5, 9, 11\}, \{2, 6\}, \{2, 7, 8\}\\
  &  \{3, 4, 7\}, \{3, 5, 8\}, \{3, 6, 11\}, \{3, 9\}, \{3, 10\}, \{4, 6, 8\}, \{4, 9\}, \{4, 11\}, \{5, 6\}, \{5, 7\}\\
  &  \{5, 10\}, \{6, 9, 10\}, \{7, 9\}, \{7, 10\}, \{7, 11\}, \{8, 10\}, \{8, 11\}
  \big\}
\end{align*}
\item The group of automorphisms of $C_\ssnum$ is trivial.

\end{itemize}

\subsection{ }

\begin{itemize}[label=--,leftmargin=*]

\item Let $\zeta$ be a fixed root of the polynomial $Z^4 + Z^3 + Z^2 + Z + 1$. For $i\in\{1,2,3,4\}$, we define $\B_\ssnum^i$ by the following equations, where $\alpha=\zeta^i$.
\begin{equation*}
  \begin{array}{l p{0.5cm} l}
    L_{1} : z = 0 &&
    L_{2} : x = 0 \\
    L_{3} : x - z = 0 &&
    L_{4} : y = 0 \\
    L_{5} : y - z = 0 &&
    L_{6} : x + (-\alpha)y = 0 \\
    L_{7} : -x + (\alpha)y + z = 0 &&
    L_{8} : (-\alpha^2 - \alpha - 1)x + (\alpha^3 + \alpha^2 + \alpha)y + z = 0 \\
    L_{9} : (-\alpha^2 - \alpha - 1)x - y + z = 0 &&
    L_{10} : (\alpha^3 + \alpha^2 + \alpha)x + (\alpha)y + z = 0 \\
    L_{11} : (\alpha^2 + \alpha)x + y = 0 &&
  \end{array}
\end{equation*}
\item The combinatorics shared by the $\B_\ssnum^i$'s is given by:
\begin{align*}
  C_\ssnum=\big\{
  & \{1, 2, 3\}, \{1, 4, 5\}, \{1, 6, 7, 8\}, \{1, 9, 10\}, \{1, 11\}, \{2, 4, 6, 11\}, \{2, 5, 9\}, \{2, 7, 10\}\\
  &  \{2, 8\}, \{3, 4, 7\}, \{3, 5\}, \{3, 6\}, \{3, 8\}, \{3, 9, 11\}, \{3, 10\}, \{4, 8, 9\}, \{4, 10\}, \{5, 6, 10\}\\
  &  \{5, 7\}, \{5, 8\}, \{5, 11\}, \{6, 9\}, \{7, 9\}, \{7, 11\}, \{8, 10, 11\}
  \big\}
\end{align*}
\item The group of automorphisms of $C_\ssnum$ is trivial.

\end{itemize}

\subsection{ }

\begin{itemize}[label=--,leftmargin=*]

\item Let $\zeta$ be a fixed root of the polynomial $Z^4 + Z^3 + Z^2 + Z + 1$. For $i\in\{1,2,3,4\}$, we define $\B_\ssnum^i$ by the following equations, where $\alpha=\zeta^i$.
\begin{equation*}
  \begin{array}{l p{0.5cm} l}
    L_{1} : z = 0 &&
    L_{2} : x = 0 \\
    L_{3} : x - z = 0 &&
    L_{4} : y = 0 \\
    L_{5} : y - z = 0 &&
    L_{6} : (\alpha^3)x + y = 0 \\
    L_{7} : -x + (\alpha)y + z = 0 &&
    L_{8} : x +(\alpha^2)y + (\alpha)z = 0 \\ 
    L_{9} : x - (\alpha)y + (\alpha)z = 0 && 
    L_{10} : (\alpha^3)x + y + (-\alpha^3 - 1)z = 0 \\
    L_{11} : (-\alpha^3 - 1)x - y + z = 0 &&
  \end{array}
\end{equation*}
\item The combinatorics shared by the $\B_\ssnum^i$'s is given by:
\begin{align*}
  C_\ssnum=\big\{
  & \{1, 2, 3\}, \{1, 4, 5\}, \{1, 6, 8, 10\}, \{1, 7, 9\}, \{1, 11\}, \{2, 4, 6\}, \{2, 5, 9, 11\}, \{2, 7, 8\}\\
  &  \{2, 10\}, \{3, 4, 7\}, \{3, 5, 10\}, \{3, 6, 11\}, \{3, 8\}, \{3, 9\}, \{4, 8, 9\}, \{4, 10\}, \{4, 11\}, \{5, 6\}\\
  &  \{5, 7\}, \{5, 8\}, \{6, 7\}, \{6, 9\}, \{7, 10, 11\}, \{8, 11\}, \{9, 10\}
  \big\}
\end{align*}
\item The group of automorphisms of $C_\ssnum$ is generated by:  $\sigma_1 = (1, 2)(5, 6)(8, 9)(10, 11)$.

\end{itemize}

\subsection{ }

\begin{itemize}[label=--,leftmargin=*]

\item Let $\zeta$ be a fixed root of the polynomial $Z^4 + 2Z^3 + 4Z^2 + 3Z + 1$. For $i\in\{1,2,3,4\}$, we define $\B_\ssnum^i$ by the following equations, where $\alpha=\zeta^i$.
\begin{equation*}
  \begin{array}{l p{0.5cm} l}
    L_{1} : z = 0 &&
    L_{2} : x = 0 \\
    L_{3} : x - z = 0 &&
    L_{4} : y = 0 \\
    L_{5} : y - z = 0 &&
    L_{6} : x + (-\alpha^3 - \alpha^2 - 2\alpha - 1)y = 0 \\
    L_{7} : (-\alpha^2 - \alpha - 2)x - y + z = 0 &&
    L_{8} : (\alpha^2+2\alpha+1)x - (\alpha^2+\alpha)y + (\alpha^2)z = 0 \\ 
    L_{9} : (\alpha+1)x - (\alpha)y - z = 0 && 
    L_{10} : (\alpha^3+\alpha^2+2\alpha)x + (\alpha+1)y - (\alpha)z = 0 \\ 
    L_{11} : (\alpha^3 + \alpha^2 + 2\alpha)x + (\alpha)y + z = 0 &&
  \end{array}
\end{equation*}
\item The combinatorics shared by the $\B_\ssnum^i$'s is given by:
\begin{align*}
  C_\ssnum=\big\{
  & \{1, 2, 3\}, \{1, 4, 5\}, \{1, 6, 10\}, \{1, 7, 11\}, \{1, 8, 9\}, \{2, 4, 6\}, \{2, 5, 7\}, \{2, 8, 10\}, \{2, 9, 11\}\\
  &  \{3, 4\}, \{3, 5, 9\}, \{3, 6\}, \{3, 7, 8\}, \{3, 10\}, \{3, 11\}, \{4, 7, 10\}, \{4, 8, 11\}, \{4, 9\}, \{5, 6, 11\}\\
  &  \{5, 8\}, \{5, 10\}, \{6, 7\}, \{6, 8\}, \{6, 9\}, \{7, 9\}, \{9, 10\}, \{10, 11\}
  \big\}
\end{align*}
\item The group of automorphisms of $C_\ssnum$ is trivial.

\end{itemize}

\subsection{ }

\begin{itemize}[label=--,leftmargin=*]

\item Let $\zeta$ be a fixed root of the polynomial $Z^4 + Z^3 + Z^2 + Z + 1$. For $i\in\{1,2,3,4\}$, we define $\B_\ssnum^i$ by the following equations, where $\alpha=\zeta^i$.
\begin{equation*}
  \begin{array}{l p{0.5cm} l}
    L_{1} : z = 0 &&
    L_{2} : x = 0 \\
    L_{3} : x - z = 0 &&
    L_{4} : y = 0 \\
    L_{5} : y - z = 0 &&
    L_{6} : (\alpha)y + z = 0\\ 
    L_{7} : x - (\alpha^3 + \alpha^2)y = 0 &&
    L_{8} : (\alpha^3 + \alpha)x - y + z = 0 \\
    L_{9} : -x - y + z = 0 &&
    L_{10} : (\alpha)x + (\alpha^2+\alpha)y + z = 0\\ 
    L_{11} : (\alpha)x + (\alpha)y + z = 0 &&
  \end{array}
\end{equation*}
\item The combinatorics shared by the $\B_\ssnum^i$'s is given by:
\begin{align*}
  C_\ssnum=\big\{
  & \{1, 2, 3\}, \{1, 4, 5, 6\}, \{1, 7\}, \{1, 8, 10\}, \{1, 9, 11\}, \{2, 4, 7\}, \{2, 5, 8, 9\}, \{2, 6, 11\}\\
  &  \{2, 10\}, \{3, 4, 9\}, \{3, 5\}, \{3, 6, 10\}, \{3, 7\}, \{3, 8\}, \{3, 11\}, \{4, 8\}, \{4, 10, 11\}, \{5, 7, 10\}\\
  &  \{5, 11\}, \{6, 7\}, \{6, 8\}, \{6, 9\}, \{7, 8, 11\}, \{7, 9\}, \{9, 10\}
  \big\}
\end{align*}
\item The group of automorphisms of $C_\ssnum$ is trivial.

\end{itemize}

\subsection{ }

\begin{itemize}[label=--,leftmargin=*]

\item Let $\zeta$ be a fixed root of the polynomial $Z^4 + Z^3 + Z^2 + Z + 1$. For $i\in\{1,2,3,4\}$, we define $\B_\ssnum^i$ by the following equations, where $\alpha=\zeta^i$.
\begin{equation*}
  \begin{array}{l p{0.5cm} l}
    L_{1} : z = 0 &&
    L_{2} : x = 0 \\
    L_{3} : x - z = 0 &&
    L_{4} : y = 0 \\
    L_{5} : y - z = 0 &&
    L_{6} : (\alpha^2)y + z = 0 \\
    L_{7} : (\alpha)x - y = 0 && 
    L_{8} : (\alpha)x - y + z = 0 \\
    L_{9} : x + (\alpha^3)y = 0 &&
    L_{10} : (\alpha^2)x + y + (-\alpha^2 - \alpha - 1)z = 0 \\
    L_{11} : x + (\alpha^3 + 1)y - z = 0 &&
  \end{array}
\end{equation*}
\item The combinatorics shared by the $\B_\ssnum^i$'s is given by:
\begin{align*}
  C_\ssnum=\big\{
  & \{1, 2, 3\}, \{1, 4, 5, 6\}, \{1, 7, 8\}, \{1, 9, 10\}, \{1, 11\}, \{2, 4, 7, 9\}, \{2, 5, 8\}, \{2, 6\}, \{2, 10, 11\}\\
  &  \{3, 4, 11\}, \{3, 5\}, \{3, 6\}, \{3, 7\}, \{3, 8, 10\}, \{3, 9\}, \{4, 8\}, \{4, 10\}, \{5, 7\}, \{5, 9, 11\}\\
  &  \{5, 10\}, \{6, 7, 10\}, \{6, 8, 11\}, \{6, 9\}, \{7, 11\}, \{8, 9\}
  \big\}
\end{align*}
\item The group of automorphisms of $C_\ssnum$ is trivial.

\end{itemize}

\subsection{ }

\begin{itemize}[label=--,leftmargin=*]

\item Let $\zeta$ be a fixed root of the polynomial $Z^4 + 2Z^3 + 4Z^2 + 3Z + 1$. For $i\in\{1,2,3,4\}$, we define $\B_\ssnum^i$ by the following equations, where $\alpha=\zeta^i$.
\begin{equation*}
  \begin{array}{l p{0.5cm} l}
    L_{1} : z = 0 &&
    L_{2} : x = 0 \\
    L_{3} : x - z = 0 &&
    L_{4} : y = 0 \\
    L_{5} : y - z = 0 &&
    L_{6} : (\alpha^3 + 2\alpha^2 + 3\alpha + 2)y - z = 0 \\
    L_{7} : (\alpha)x + y = 0 &&
    L_{8} : (\alpha^2 + \alpha + 1)x - (\alpha^3 + 2\alpha^2 + 3\alpha + 2)y + z = 0 \\
    L_{9} : (\alpha+1)x - (\alpha)y = 0 && 
    L_{10} : (\alpha + 1)x - (\alpha)y - z = 0 \\
    L_{11} : (\alpha + 1)x + y - z = 0 &&
  \end{array}
\end{equation*}
\item The combinatorics shared by the $\B_\ssnum^i$'s is given by:
\begin{align*}
  C_\ssnum=\big\{
  & \{1, 2, 3\}, \{1, 4, 5, 6\}, \{1, 7, 8\}, \{1, 9, 10\}, \{1, 11\}, \{2, 4, 7, 9\}, \{2, 5, 11\}, \{2, 6, 8\}\\
  &  \{2, 10\}, \{3, 4\}, \{3, 5, 10\}, \{3, 6\}, \{3, 7, 11\}, \{3, 8\}, \{3, 9\}, \{4, 8\}, \{4, 10, 11\}, \{5, 7\}\\
  &  \{5, 8, 9\}, \{6, 7, 10\}, \{6, 9\}, \{6, 11\}, \{8, 10\}, \{8, 11\}, \{9, 11\}
  \big\}
\end{align*}
\item The group of automorphisms of $C_\ssnum$ is trivial.

\end{itemize}

\subsection{ }

\begin{itemize}[label=--,leftmargin=*]

\item Let $\zeta$ be a fixed root of the polynomial $Z^4 + Z^3 + Z^2 + Z + 1$. For $i\in\{1,2,3,4\}$, we define $\B_\ssnum^i$ by the following equations, where $\alpha=\zeta^i$.
\begin{equation*}
  \begin{array}{l p{0.5cm} l}
    L_{1} : z = 0 &&
    L_{2} : x = 0 \\
    L_{3} : x - z = 0 &&
    L_{4} : y = 0 \\
    L_{5} : y - z = 0 &&
    L_{6} : y + (\alpha^3 + \alpha^2)z = 0 \\
    L_{7} : x - (\alpha^3+\alpha+1)y - z = 0 && 
    L_{8} : (-\alpha^3 - \alpha^2 - \alpha)x + y + (\alpha^3 + \alpha^2)z = 0 \\
    L_{9} : (-\alpha)x + y = 0 &&
    L_{10} : (\alpha)x - y + z = 0 \\
    L_{11} : x + (\alpha^3 + \alpha)y = 0 &&
  \end{array}
\end{equation*}
\item The combinatorics shared by the $\B_\ssnum^i$'s is given by:
\begin{align*}
  C_\ssnum=\big\{
  & \{1, 2, 3\}, \{1, 4, 5, 6\}, \{1, 7, 8\}, \{1, 9, 10\}, \{1, 11\}, \{2, 4, 9, 11\}, \{2, 5, 10\}, \{2, 6, 8\}\\
  &  \{2, 7\}, \{3, 4, 7\}, \{3, 5\}, \{3, 6\}, \{3, 8, 9\}, \{3, 10, 11\}, \{4, 8\}, \{4, 10\}, \{5, 7, 11\}, \{5, 8\}\\
  &  \{5, 9\}, \{6, 7, 10\}, \{6, 9\}, \{6, 11\}, \{7, 9\}, \{8, 10\}, \{8, 11\}
  \big\}
\end{align*}
\item The group of automorphisms of $C_\ssnum$ is trivial.

\end{itemize}

\subsection{ }

\begin{itemize}[label=--,leftmargin=*]

\item Let $\zeta$ be a fixed root of the polynomial $Z^4 + Z^3 + Z^2 + Z + 1$. For $i\in\{1,2,3,4\}$, we define $\B_\ssnum^i$ by the following equations, where $\alpha=\zeta^i$.
\begin{equation*}
  \begin{array}{l p{0.5cm} l}
    L_{1} : z = 0 &&
    L_{2} : x = 0 \\
    L_{3} : x - z = 0 &&
    L_{4} : y = 0 \\
    L_{5} : y - z = 0 &&
    L_{6} : (-\alpha^3 - 1)y + z = 0 \\
    L_{7} : (-\alpha^3)x - y + z = 0 &&
    L_{8} : x + (\alpha^2)y + (\alpha)z = 0 \\ 
    L_{9} : x - (\alpha)y + (\alpha)z = 0 && 
    L_{10} : -x + (\alpha)y + z = 0 \\
    L_{11} : -x + y = 0 &&
  \end{array}
\end{equation*}
\item The combinatorics shared by the $\B_\ssnum^i$'s is given by:
\begin{align*}
  C_\ssnum=\big\{
  & \{1, 2, 3\}, \{1, 4, 5, 6\}, \{1, 7, 8\}, \{1, 9, 10\}, \{1, 11\}, \{2, 4, 11\}, \{2, 5, 7, 9\}, \{2, 6\}, \{2, 8, 10\}\\
  &  \{3, 4, 10\}, \{3, 5, 11\}, \{3, 6, 8\}, \{3, 7\}, \{3, 9\}, \{4, 7\}, \{4, 8, 9\}, \{5, 8\}, \{5, 10\}, \{6, 7, 11\}\\
  &  \{6, 9\}, \{6, 10\}, \{7, 10\}, \{8, 11\}, \{9, 11\}, \{10, 11\}
  \big\}
\end{align*}
\item The group of automorphisms of $C_\ssnum$ is trivial.

\end{itemize}

\subsection{ }

\begin{itemize}[label=--,leftmargin=*]

\item Let $\zeta$ be a fixed root of the polynomial $Z^4 + Z^3 + Z^2 + Z + 1$. For $i\in\{1,2,3,4\}$, we define $\B_\ssnum^i$ by the following equations, where $\alpha=\zeta^i$.
\begin{equation*}
  \begin{array}{l p{0.5cm} l}
    L_{1} : z = 0 &&
    L_{2} : x = 0 \\
    L_{3} : x - z = 0 &&
    L_{4} : y = 0 \\
    L_{5} : y - z = 0 &&
    L_{6} : x + (\alpha)y = 0 \\ 
    L_{7} : y + (-\alpha - 1)z = 0 &&
    L_{8} : (-\alpha^2 - \alpha - 1)x + y = 0 \\
    L_{9} : (\alpha)x + (\alpha^2)y + z = 0 &&
    L_{10} : (\alpha)x - y + z = 0 \\
    L_{11} : (-\alpha^2 - \alpha - 1)x + y + (-\alpha - 1)z = 0 &&
  \end{array}
\end{equation*}
\item The combinatorics shared by the $\B_\ssnum^i$'s is given by:
\begin{align*}
  C_\ssnum=\big\{
  & \{1, 2, 3\}, \{1, 4, 5, 7\}, \{1, 6, 9\}, \{1, 8, 11\}, \{1, 10\}, \{2, 4, 6, 8\}, \{2, 5, 10\}, \{2, 7, 11\}\\
  &  \{2, 9\}, \{3, 4\}, \{3, 5\}, \{3, 6\}, \{3, 7, 10\}, \{3, 8, 9\}, \{3, 11\}, \{4, 9, 10\}, \{4, 11\}, \{5, 6\}\\
  &  \{5, 8\}, \{5, 9, 11\}, \{6, 7\}, \{6, 10, 11\}, \{7, 8\}, \{7, 9\}, \{8, 10\}
  \big\}
\end{align*}
\item The group of automorphisms of $C_\ssnum$ is generated by:  $\sigma_1 = (1, 2)(5, 6)(7, 8)(9, 10)$.

\end{itemize}

\subsection{ }

\begin{itemize}[label=--,leftmargin=*]

\item Let $\zeta$ be a fixed root of the polynomial $Z^4 + Z^3 + Z^2 + Z + 1$. For $i\in\{1,2,3,4\}$, we define $\B_\ssnum^i$ by the following equations, where $\alpha=\zeta^i$.
\begin{equation*}
  \begin{array}{l p{0.5cm} l}
    L_{1} : z = 0 &&
    L_{2} : x = 0 \\
    L_{3} : x - z = 0 &&
    L_{4} : y = 0 \\
    L_{5} : y - z = 0 &&
    L_{6} : (\alpha)x + y = 0 \\ 
    L_{7} : y + (-\alpha^2 - 1)z = 0 &&
    L_{8} : (-\alpha^3 - \alpha^2 - 1)x + y = 0 \\
    L_{9} : (-\alpha^3)x + (\alpha)y + z = 0 &&
    L_{10} : (\alpha^2)x + (\alpha)y + z = 0 \\
    L_{11} : (\alpha^2)x - y + z = 0 &&
  \end{array}
\end{equation*}
\item The combinatorics shared by the $\B_\ssnum^i$'s is given by:
\begin{align*}
  C_\ssnum=\big\{
  & \{1, 2, 3\}, \{1, 4, 5, 7\}, \{1, 6, 10\}, \{1, 8\}, \{1, 9, 11\}, \{2, 4, 6, 8\}, \{2, 5, 11\}, \{2, 7\}, \{2, 9, 10\}\\
  &  \{3, 4\}, \{3, 5\}, \{3, 6\}, \{3, 7, 11\}, \{3, 8, 10\}, \{3, 9\}, \{4, 9\}, \{4, 10, 11\}, \{5, 6\}, \{5, 8, 9\}\\
  &  \{5, 10\}, \{6, 7, 9\}, \{6, 11\}, \{7, 8\}, \{7, 10\}, \{8, 11\}
  \big\}
\end{align*}
\item The group of automorphisms of $C_\ssnum$ is generated by:  $\sigma_1 = (1, 2)(5, 6)(7, 8)(10, 11)$.

\end{itemize}

\subsection{ }

\begin{itemize}[label=--,leftmargin=*]

\item Let $\zeta$ be a fixed root of the polynomial $Z^4 + Z^3 + Z^2 + Z + 1$. For $i\in\{1,2,3,4\}$, we define $\B_\ssnum^i$ by the following equations, where $\alpha=\zeta^i$.
\begin{equation*}
  \begin{array}{l p{0.5cm} l}
    L_{1} : z = 0 &&
    L_{2} : x = 0 \\
    L_{3} : x - z = 0 &&
    L_{4} : (\alpha)x + z = 0\\ 
    L_{5} : y = 0 &&
    L_{6} : y - z = 0 \\
    L_{7} : -x + (\alpha^3 + \alpha^2 + \alpha)y + z = 0 &&
    L_{8} : (-\alpha^3 - 1)x + (\alpha)y + z = 0 \\
    L_{9} : (2\alpha^3 + \alpha + 1)x + y + (-\alpha^3 - \alpha - 1)z = 0 &&
    L_{10} : (-\alpha^3 - 1)x - y + z = 0 \\
    L_{11} : (-\alpha^3 - \alpha^2 - 1)x + (\alpha)y + z = 0 &&
  \end{array}
\end{equation*}
\item The combinatorics shared by the $\B_\ssnum^i$'s is given by:
\begin{align*}
  C_\ssnum=\big\{
  & \{1, 2, 3, 4\}, \{1, 5, 6\}, \{1, 7, 8\}, \{1, 9\}, \{1, 10, 11\}, \{2, 5\}, \{2, 6, 10\}, \{2, 7, 9\}, \{2, 8, 11\}\\
  &  \{3, 5, 7\}, \{3, 6\}, \{3, 8\}, \{3, 9, 10\}, \{3, 11\}, \{4, 5\}, \{4, 6, 7\}, \{4, 8, 9\}, \{4, 10\}, \{4, 11\}\\
  &  \{5, 8, 10\}, \{5, 9, 11\}, \{6, 8\}, \{6, 9\}, \{6, 11\}, \{7, 10\}, \{7, 11\}
  \big\}
\end{align*}
\item The group of automorphisms of $C_\ssnum$ is trivial.

\end{itemize}

\subsection{ }

\begin{itemize}[label=--,leftmargin=*]

\item Let $\zeta$ be a fixed root of the polynomial $Z^4 + 3Z^3 + 4Z^2 + 2Z + 1$. For $i\in\{1,2,3,4\}$, we define $\B_\ssnum^i$ by the following equations, where $\alpha=\zeta^i$.
\begin{equation*}
  \begin{array}{l p{0.5cm} l}
    L_{1} : z = 0 &&
    L_{2} : x = 0 \\
    L_{3} : x - z = 0 &&
    L_{4} : (-\alpha^3 - 3\alpha^2 - 3\alpha - 1)x + z = 0 \\
    L_{5} : y = 0 &&
    L_{6} : y - z = 0 \\
    L_{7} : (-\alpha - 1)x + y = 0 &&
    L_{8} : (\alpha+1)x - y - (\alpha)z = 0\\ 
    L_{9} : (\alpha)x - y + z = 0 &&
    L_{10} : (\alpha^2+\alpha)x - (\alpha+1)y + (\alpha^2)z = 0 \\ 
    L_{11} : (\alpha^2)x + y + (\alpha)z = 0 &&
  \end{array}
\end{equation*}
\item The combinatorics shared by the $\B_\ssnum^i$'s is given by:
\begin{align*}
  C_\ssnum=\big\{
  & \{1, 2, 3, 4\}, \{1, 5, 6\}, \{1, 7, 8\}, \{1, 9, 10\}, \{1, 11\}, \{2, 5, 7\}, \{2, 6, 9\}, \{2, 8, 11\}, \{2, 10\}\\
  &  \{3, 5\}, \{3, 6, 8\}, \{3, 7, 9\}, \{3, 10\}, \{3, 11\}, \{4, 5\}, \{4, 6, 10\}, \{4, 7, 11\}, \{4, 8\}, \{4, 9\}\\
  &  \{5, 8, 10\}, \{5, 9, 11\}, \{6, 7\}, \{6, 11\}, \{7, 10\}, \{8, 9\}, \{10, 11\}
  \big\}
\end{align*}
\item The group of automorphisms of $C_\ssnum$ is generated by:  $\sigma_1 = (1, 2)(6, 7)(8, 9)(10, 11)$.

\end{itemize}

\subsection{ }

\begin{itemize}[label=--,leftmargin=*]

\item Let $\zeta$ be a fixed root of the polynomial $Z^4 - Z^3 + Z^2 - Z + 1$. For $i\in\{1,2,3,4\}$, we define $\B_\ssnum^i$ by the following equations, where $\alpha=\zeta^i$.
\begin{equation*}
  \begin{array}{l p{0.5cm} l}
    L_{1} : z = 0 &&
    L_{2} : x = 0 \\
    L_{3} : x - z = 0 &&
    L_{4} : (\alpha)x + z = 0 \\
    L_{5} : y = 0 &&
    L_{6} : y - z = 0 \\
    L_{7} : -x - y + z = 0 &&
    L_{8} : x + y + (\alpha - 1)z = 0 \\
    L_{9} : (\alpha)x + y = 0 &&
    L_{10} : (\alpha)x + y + (\alpha^2-\alpha)z = 0\\ 
    L_{11} : x + (-\alpha^2 + \alpha)y + (-\alpha^3 + 2\alpha^2 - \alpha)z = 0 &&
  \end{array}
\end{equation*}
\item The combinatorics shared by the $\B_\ssnum^i$'s is given by:
\begin{align*}
  C_\ssnum=\big\{
  & \{1, 2, 3, 4\}, \{1, 5, 6\}, \{1, 7, 8\}, \{1, 9, 10\}, \{1, 11\}, \{2, 5, 9\}, \{2, 6, 7\}, \{2, 8, 11\}, \{2, 10\}\\
  &  \{3, 5, 7\}, \{3, 6\}, \{3, 8, 9\}, \{3, 10\}, \{3, 11\}, \{4, 5\}, \{4, 6, 9\}, \{4, 7, 11\}, \{4, 8\}, \{4, 10\}\\
  &  \{5, 8, 10\}, \{5, 11\}, \{6, 8\}, \{6, 10, 11\}, \{7, 9\}, \{7, 10\}, \{9, 11\}
  \big\}
\end{align*}
\item The group of automorphisms of $C_\ssnum$ is trivial.

\end{itemize}

\subsection{ }

\begin{itemize}[label=--,leftmargin=*]

\item Let $\zeta$ be a fixed root of the polynomial $Z^4 + Z^3 + Z^2 + Z + 1$. For $i\in\{1,2,3,4\}$, we define $\B_\ssnum^i$ by the following equations, where $\alpha=\zeta^i$.
\begin{equation*}
  \begin{array}{l p{0.5cm} l}
    L_{1} : z = 0 &&
    L_{2} : x = 0 \\
    L_{3} : x - z = 0 &&
    L_{4} : (\alpha^3)x + z = 0 \\
    L_{5} : y = 0 &&
    L_{6} : y - z = 0 \\
    L_{7} : (\alpha^3)x + y = 0 &&
    L_{8} : x - (\alpha)y + (\alpha)z = 0\\ 
    L_{9} : -x + (\alpha)y + z = 0 &&
    L_{10} : x + (\alpha^2)y + (\alpha)z = 0\\ 
    L_{11} : (\alpha^3 + \alpha^2)x + y + (-\alpha^3 - \alpha^2 - 1)z = 0 &&
  \end{array}
\end{equation*}
\item The combinatorics shared by the $\B_\ssnum^i$'s is given by:
\begin{align*}
  C_\ssnum=\big\{
  & \{1, 2, 3, 4\}, \{1, 5, 6\}, \{1, 7, 10\}, \{1, 8, 9\}, \{1, 11\}, \{2, 5, 7\}, \{2, 6, 8\}, \{2, 9, 10\}, \{2, 11\}\\
  &  \{3, 5, 9\}, \{3, 6, 11\}, \{3, 7\}, \{3, 8\}, \{3, 10\}, \{4, 5\}, \{4, 6, 7\}, \{4, 8, 11\}, \{4, 9\}, \{4, 10\}\\
  &  \{5, 8, 10\}, \{5, 11\}, \{6, 9\}, \{6, 10\}, \{7, 8\}, \{7, 9, 11\}, \{10, 11\}
  \big\}
\end{align*}
\item The group of automorphisms of $C_\ssnum$ is trivial.

\end{itemize}

\subsection{ }

\begin{itemize}[label=--,leftmargin=*]

\item Let $\zeta$ be a fixed root of the polynomial $Z^4 + Z^3 + Z^2 + Z + 1$. For $i\in\{1,2,3,4\}$, we define $\B_\ssnum^i$ by the following equations, where $\alpha=\zeta^i$.
\begin{equation*}
  \begin{array}{l p{0.5cm} l}
    L_{1} : z = 0 &&
    L_{2} : x = 0 \\
    L_{3} : y = 0 &&
    L_{4} : x - z = 0 \\
    L_{5} : y - z = 0 &&
    L_{6} : x - (\alpha)y + (\alpha)z = 0\\ 
    L_{7} : -x + (\alpha)y + z = 0 &&
    L_{8} : x + (\alpha^2)y + (\alpha)z = 0 \\ 
    L_{9} : (\alpha^2 + \alpha)x + (-\alpha^2 - \alpha - 1)y + z = 0 &&
    L_{10} : (2\alpha^3 + \alpha^2 + \alpha + 1)x + y + (-\alpha^3 - 1)z = 0 \\
    L_{11} : x + (\alpha^2-\alpha)y - (\alpha^2-1)z = 0&& 
  \end{array}
\end{equation*}
\item The combinatorics shared by the $\B_\ssnum^i$'s is given by:
\begin{align*}
  C_\ssnum=\big\{
  & \{1, 2, 4\}, \{1, 3, 5\}, \{1, 6, 7\}, \{1, 8, 9\}, \{1, 10\}, \{1, 11\}, \{2, 3\}, \{2, 5, 6\}, \{2, 7, 8\}, \{2, 9, 10\}\\
  &  \{2, 11\}, \{3, 4, 7\}, \{3, 6, 8\}, \{3, 9, 11\}, \{3, 10\}, \{4, 5, 9\}, \{4, 6, 10\}, \{4, 8\}, \{4, 11\}\\
  &  \{5, 7, 11\}, \{5, 8\}, \{5, 10\}, \{6, 9\}, \{6, 11\}, \{7, 9\}, \{7, 10\}, \{8, 10, 11\}
  \big\}
\end{align*}
\item The group of automorphisms of $C_\ssnum$ is generated by:  $\sigma_1 = (2, 3)(4, 5)(6, 7)(10, 11)$.

\end{itemize}

\subsection{ }

\begin{itemize}[label=--,leftmargin=*]

\item Let $\zeta$ be a fixed root of the polynomial $Z^4 + 3Z^3 + 4Z^2 + 2Z + 1$. For $i\in\{1,2,3,4\}$, we define $\B_\ssnum^i$ by the following equations, where $\alpha=\zeta^i$.
\begin{equation*}
  \begin{array}{l p{0.5cm} l}
    L_{1} : z = 0 &&
    L_{2} : x = 0 \\
    L_{3} : y = 0 &&
    L_{4} : (-\alpha - 1)x + (\alpha)y + z = 0 \\
    L_{5} : x - z = 0 &&
    L_{6} : (-\alpha - 1)x + (\alpha^3 + 2\alpha^2 + 2\alpha)y + z = 0 \\
    L_{7} : y - z = 0 &&
    L_{8} : (\alpha+1)x - (\alpha)y + (\alpha^2+\alpha)z = 0 \\ 
    L_{9} : (\alpha^2 + 2\alpha + 1)x + y = 0 &&
    L_{10} : -x + (\alpha)y + z = 0 \\
    L_{11} : x - (\alpha)y + (\alpha)z = 0 && 
  \end{array}
\end{equation*}
\item The combinatorics shared by the $\B_\ssnum^i$'s is given by:
\begin{align*}
  C_\ssnum=\big\{
  & \{1, 2, 5\}, \{1, 3, 7\}, \{1, 4, 8\}, \{1, 6, 9\}, \{1, 10, 11\}, \{2, 3, 9\}, \{2, 4, 10\}, \{2, 6, 8\}, \{2, 7, 11\}\\
  &  \{3, 4, 6\}, \{3, 5, 10\}, \{3, 8, 11\}, \{4, 5, 7\}, \{4, 9, 11\}, \{5, 6\}, \{5, 8\}, \{5, 9\}, \{5, 11\}, \{6, 7\}\\
  &  \{6, 10\}, \{6, 11\}, \{7, 8\}, \{7, 9\}, \{7, 10\}, \{8, 9\}, \{8, 10\}, \{9, 10\}
  \big\}
\end{align*}
\item The group of automorphisms of $C_\ssnum$ is generated by:  $\sigma_1 = (1, 4, 2, 3)(5, 6)(7, 8, 10, 9)$.

\end{itemize}

\end{spacing}


\newpage
\bibliographystyle{plain}
\bibliography{biblio}

\end{document}